\documentclass[11pt]{article}
\usepackage[top=1in,bottom=1in,left=1in,right=1in]{geometry}
\usepackage{amsmath,amsfonts,amsthm}
\usepackage{bm}
\usepackage{color}
\usepackage{hyperref}
\hypersetup{colorlinks=true,
            linkcolor={blue},
            citecolor={falured}}
\theoremstyle{definition}
\newtheorem{assumption}{Assumption}
\usepackage{amsmath,amssymb,fullpage}
\usepackage{enumerate}
\usepackage{graphicx}
\usepackage{subfig}
\usepackage{makeidx}
\usepackage[utf8]{inputenc}
\usepackage{wrapfig}
\usepackage{mathrsfs}
\usepackage{amssymb}
\usepackage{latexsym}
\usepackage{amsbsy}
\usepackage{bbm}
\usepackage{mathrsfs}
\usepackage{todonotes}
\usepackage{url}
\usepackage{float}

\usepackage{wasysym}
\def\email#1{\it #1\par}

\providecommand{\otherindexspace}[1]{}

\usepackage{amsthm}
\newtheorem{theorem}{Theorem}[section]
\newtheorem{lemma}[theorem]{Lemma}
\newtheorem{proposition}[theorem]{Proposition}
\newtheorem{remark}[theorem]{Remark}
\newtheorem{definition}[theorem]{Definition}

\newtheorem{corollary}[theorem]{Corollary}

\DeclareMathOperator*{\argmax}{arg\,max}
\DeclareMathOperator{\op}{\overline{\mathbb P}}

\definecolor{falured}{rgb}{0.5, 0.09, 0.09}

\usepackage[ruled,vlined, linesnumbered]{algorithm2e}

\usepackage{comment}

\usepackage{dsfont}

\usepackage{biblatex}
\numberwithin{equation}{section}

\title{Energy transition under scenario uncertainty: a mean-field game of stopping with common noise}

\author{ Roxana Dumitrescu \thanks{Department of Mathematics, King's College London, Strand, London, WC2R 2LS, United Kingdom, Email: \email roxana.dumitrescu@kcl.ac.uk} \and Marcos Leutscher  \thanks{CREST, ENSAE, Institut Polytechnique de Paris, 5 avenue Henry Le Chatelier, 91120 Palaiseau, France, Email: \email marcos.leutscherdelasnieves@ensae.fr} \and Peter Tankov \thanks{CREST, ENSAE, Institut Polytechnique de Paris, 5 avenue Henry Le Chatelier, 91120 Palaiseau, France, Email: \email peter.tankov@ensae.fr}}

\date{}

\begin{document}
\maketitle

\begin{abstract}
We study the impact of transition scenario uncertainty, namely that of future carbon price and electricity demand, on the pace of decarbonization of the electricity industry. 
To this end, we develop a theory of optimal stopping mean-field games with non-Markovian common noise and partial observation. For mathematical tractability, the theory is formulated in discrete time and with common noise restricted to a finite probability space. We prove the existence of Nash equilibria for this game using the linear programming approach. 
We then apply the general theory to build a discrete time  model for the long-term dynamics of the electricity market subject to common random shocks affecting the carbon price and the electricity demand.  We consider two classes of agents: conventional producers and renewable producers. The former choose an optimal moment to exit the market and the latter choose an optimal moment to enter the market by investing into renewable generation. The agents interact through the market price determined by a merit order mechanism with an exogenous stochastic demand.  We illustrate our model by an example inspired by the UK electricity market, and show that scenario uncertainty leads to significant changes in the speed of replacement of conventional generators by renewable production.
\end{abstract}

\textbf{Key words}: Electricity market, energy transition, scenario uncertainty, mean-field games, common noise, optimal stopping, partial information\\

\textbf{MSC Classification}: 91A55, 91A13, 91A80\\

\textbf{JEL Classification}: C73, Q42

\section{Introduction}
Given the climate emergency, there is no doubt that in the years and decades to come, the transition to a low-carbon economy will  lead to fundamental transformations in the energy industry. However, there is a considerable uncertainty about the pace of these transformations, which is only slowly resolved through government announcements and regulatory shifts. Economic agents operating in energy markets must therefore make their investment decisions taking into account the uncertainty about the future market shares, costs and profits of various electricity generation technologies.  In the face of these uncertainties, the scenario approach has emerged as a means to structure decision making and optimize future actions. Scenarios are plausible trajectories of evolution of macroeconomic variables, parameterized by specific assumptions on future climate and policy, which are typically produced with integrated assessment models (IAM) and maintained by international organizations such as IEA (International Energy Agency), IPCC (Intergovernmental Panel on Climate Change) and NGFS (Network for Greening the Financial System). These databases typically contain many scenarios, which differ both in the underlying IAM and the relative stringency of climate policy assumptions. 
For example, the NGFS database\footnote{Available at \url{https://data.ene.iiasa.ac.at/ngfs/}} contains 6 sets of scenarios with varying stringency, from Current Policies (least stringent) to Net Zero 2050 (most stringent). Thus, even if a specific scenario database has been fixed, decision makers still face a considerable scenario uncertainty regarding the future policy choices.

Our aim in this paper is to understand how the scenario uncertainty impacts the dynamics of the electrical industry, and, in particular, the rate at which conventional generation is replaced by renewable plants, in the presence of many interacting agents. To achieve this goal, we extend the electricity market model based on mean-field games (MFGs) of optimal stopping of \cite{adt2021} by incorporating scenario uncertainty.

To this end, we build a theory of optimal stopping MFGs with common noise and partial information, which was not previously available in the literature. Following, \cite{bdt2020,dlt2021,adt2021} we employ the \emph{linear programming approach}, which is a compactification technique that works by reformulating the problem in terms of occupation measures of the agents instead of stopping times. 
Starting with a partial information setting, we recover Markovian dynamics by enlarging the state space. This allows us to show the existence of strong solutions in MFGs with common noise using topological arguments and also to build numerical algorithms, which are instrumental for practical applications. Here strong solution should be understood in the sense of \cite{carmona2016a}, i.e. the mean-field terms are adapted to the common noise filtration. We give a complete treatment of MFGs with optimal stopping {in a discrete time setting and involving a common noise taking values in a finite set}. This includes an interpretation of the occupation measures in terms of randomized stopping times (see \cite{baxter1977}), as well as a rigorous construction of approximate Nash equilibria for games with finite number of players. { In the context of MFGs of optimal stopping of preemption type, under quite strong assumptions, an approximation result with distributed strategies for finite player games is provided in \cite{carmona2017}. To the best of our knowledge, in the setting of MFGs of optimal stopping of war of attrition type, our paper is the first to show such an approximation result.} Finally, we propose a numerical algorithm based on the linear programming fictitious play introduced in \cite{dlt2022}, which is adapted to the framework of common noise. 

We then apply the general theory to build a discrete time  model for the long-term dynamics of the electricity market subject to common random shocks affecting the carbon price and the electricity demand. 
Following \cite{adt2021}, we assume that two types of agents, conventional producers and renewable producers, are interacting through the electricity price. The conventional producers aim to find the optimal time to leave the market, and the potential renewable project owners aim to find the optimal time to invest and enter the market. Unlike \cite{adt2021}, where the future parameters of the market  are assumed to be deterministic and known to the agents, here, we suppose that the electricity demand and the costs of the conventional producers depend on the random carbon price which is influenced by the government announcements / regulatory changes.

We illustrate our model with an example inspired by the UK electricity market, and show that scenario uncertainty leads to significant changes in the speed of replacement of conventional generators by renewable production, emphasizing the role of reliable information for successful energy transition.

\paragraph{Literature review} A variety of approaches have been used to analyse the dynamics of environmental transition of the energy sector depending on carbon price levels and technology policies. Transition scenarios used by international organisms such as IPCC, NGFS or IEA are produced with \emph{integrated assessment models}: partial or general equilibrium macroeconomic optimization models with a detailed representation of the energy sector, such as REMIND \cite{bauer2012remind}, IMACLIM \cite{hourcade2010imaclim}, TIMES \cite{loulou2008etsap} etc., or with energy system optimization models, which focus specifically on the energy sector and determine the lowest cost trajectory under a set of assumptions \cite{decarolis2017formalizing}. These models in most cases produce a single deterministic scenario, usually by solving the optimization problem of a central planner with perfect foresight. Interaction among agents, imperfect information, and uncertainty about the future values of demand, technology costs and other factors are rarely taken into account. At the other end of the spectrum one finds the computational agent-based models \cite{WV08}. These models allow for heterogeneous agents and a precise description of their interactions and of the market structure, but involve very intensive computations and do not provide any insight about the model (uniqueness of the equilibrium, robustness etc.) beyond what can be recovered from a simulated trajectory.

The MFGs are a viable compromise between the complexity of computational agent-based models and the tractability of
fully analytic approaches. MFGs, introduced in \cite{lasry2006a, lasry2006b, LL07, huang2006} are
stochastic games with a large number of identical agents in statistical sense and symmetric
interactions where each agent interacts with the average density of
the other agents (the mean field) rather than with each individual
agent. This simplifies the problem, leading to explicit solutions or
efficient numerical methods for computing the equilibrium dynamics. In
the recent years MFGs have been successfully used to model specific
sectors of electricity markets, such as price formation \cite{gomes2020mean,feron2022price}, electric vehicles \cite{CPTD12,shokri2018mean,tchuendom2019quantilized}, demand dispatch \cite{BB14,kizilkale2019integral,elie2021mean}, storage \cite{AMT20,ACDZ21}, pollution regulation \cite{santibanez2023pollution,carmona2022mean}  and renewable energy certificates markets \cite{shrivats2022mean}. 

An important recent development is the introduction of
MFGs of optimal stopping \cite{B18,bdt2020,carmona2017,gomes2015obstacle,dlt2021,nutz2018mean,he2023mean,possamai2023mean}, which can
describe technology switches and entry/exit decisions of players.

In \cite{adt2021}, the authors used the linear programming approach of \cite{bdt2020} to build an MFG model for the long-term evolution of an electricity market allowing for two classes of agents (conventional and renewable) interacting through the market price. That paper, similarly to most of the literature on MFGs, assumes that the sources of randomness affecting different agents are independent, and averaged out in the mean-field limit. As a result, the model produces a deterministic price trajectory for a given deterministic scenario of electricity demand and carbon price. However, the future climate, climate-related economic policies, and therefore also future energy demand and carbon price are subject to deep uncertainties \cite{chenet2019climate,monasterolo2019uncertainty,bolton2020green,nordhaus2018projections}. These are related, among other factors, to uncertainty about future availability of mitigation technologies  and international cooperation \cite{edenhofer2006induced}; uncertainty of climate sensitivity and carbon budgets, which may have to be revised in the future \cite{meehl2020context,cox2018emergent}, uncertainty about tipping points, which may require bold immediate actions \cite{keen2022estimates,weitzman2009modeling}, etc. In this context, perfect knowledge of the scenario corresponding to a fixed climate objective is clearly not a valid assumption. In this paper, we are therefore interested in the impact of uncertainty affecting the future carbon price and electricity demand on the pace of energy transition. To this end, we extend the model of \cite{adt2021} by allowing for common noise affecting all agents, which is not averaged out in the mean-field limit. 

Although MFGs with common noise were introduced already in \cite{lions2007theorie}, and early papers contain examples of solvable settings \cite{gueant2011mean}, the general analysis of MFGs with common noise (with regular controls) was only presented in \cite{carmona2016a}. The main difficulty is due to the much larger dimension of the natural space for the main objects of the MFG problem. For example, in the case of regular controls, the time-dependent occupation measure of the representative agent typically lives in the space of continuous functions from the time interval to the space of probability measures $\mathcal C([0,T],\mathcal P(\mathbb R^d))$; in the presence of common noise one needs to consider the space $[\mathcal C([0,T],\mathcal P(\mathbb R^d))]^\Omega$, where $\Omega$ is the probability space carrying the common noise: this makes it nearly impossible to apply the usual compactness arguments. {As pointed out in \cite{carmona2016a}, the operation of conditioning fails to be continuous in any useful sense, which complicates the study of fixed points via topological arguments.} The situation simplifies in the setting of finite probability spaces, see \cite{belak2021continuous}. A promising approach for the special class of \emph{supermodular} mean-field games allowing for a simplified treatment of common noise was recently proposed in \cite{dianetti2021submodular}. {Very recently \cite{he2023mean}, developed a different approach based on a mean-field version of the Bank-El Karoui representation result for stochastic processes, which also allows for common noise. }

In the context of MFGs of optimal stopping, which are the object of this paper, introducing common noise is even more difficult, since occupation measures of agents are not continuous as function of time. The models of \cite{carmona2017,nutz2018mean} do allow for common noise and establish the existence of an equilibrium using the special structure of the game (e.g., the complementarity property which is characteristic of preemption games, appearing e.g., in bank run problems). Our setting is that of games of war of attrition, and the methods of  \cite{carmona2017,nutz2018mean} are therefore not applicable. {Similarly, \cite{dianetti2021submodular}
 and \cite{he2023mean} impose a monotonicity condition in the measure argument on the reward function (see Assumption 7.4 in \cite{dianetti2021submodular} and Proposition 2.20 in \cite{he2023mean}), which excludes the type of games considered in this paper. The present paper is therefore, to the best of our knowledge, the first attempt to include common noise into optimal stopping mean field games of war of attrition type.}

Furthermore, in a general setting with common noise in \cite{carmona2016a,carmona2017}, the authors give abstract existence results using the notion of weak solutions, which do not seem convenient for numerical algorithms. For these reasons, to include common noise into the problem, we place ourselves in a discrete-time and finite probability space framework for the common noise. We thus develop a linear programming formulation for discrete-time optimal stopping MFGs with common noise, where the state process is a Markov chain but the common noise may be non-Markovian, showing the existence of Nash equilibrium using Kakutani-Fan-Glicksberg fixed-point theorem for set valued mappings. Note that a recent preprint \cite{guo2022mf} discusses the linear programming formulation for discrete-time MFG based on controlled Markov chains, but these authors do not consider the setting of optimal stopping, common noise and partial information.

{ The paper is organized as follows.  In Section \ref{abstract}, we define the MFG problem in a general framework. We show existence of equilibria, give a probabilistic representation of our constraint and provide an approximation result for finite player games. In Section \ref{sec n player}, we describe the electricity market model in the finite player game setting, then, in Section \ref{model_mfg}, we show how the limiting formulation of this model fits in the framework of Section \ref{abstract}, which allows to derive existence of equilibria and the uniqueness of the equilibrium price process. In Section \ref{numerics} we describe an adaptation of the algorithm proposed in \cite{dlt2022} to the common noise case. Finally, an illustration of the model, inspired by the UK electricity market, is provided in Section \ref{illustration}.
In the Appendix we give some complementary results on the linear programming approach and other technical results.}

\paragraph{Notation.} Throughout this paper, empty sums are considered to be zero. For a topological space $(E, \tau)$ we denote by $\mathcal{B}(E)$ the Borel $\sigma$-algebra, by $\mathcal{M}^s(E)$ the set of Borel finite signed measures on $E$, by $\mathcal{M}(E)$ the set of Borel finite positive measures on $E$, by $\mathcal{P}^{sub}(E)$ the set of Borel subprobability measures on $E$ and by $\mathcal{P}(E)$ the set of Borel probability measures on $E$. We denote by $M(E)$ the set of Borel measurable functions from $E$ to $\mathbb R$, by $M_b(E)$ the subset of Borel measurable and bounded functions, by $C(E)$ the subset of continuous functions, and by $C_b(E)$ the subset of continuous and bounded functions. The set $M_b(E)$ is endowed with the supremum norm $\|\varphi\|_\infty=\sup_{x\in E}|\varphi(x)|$. Given a probability space $(\Omega, \mathcal{F}, \mathbb P)$ and a sub sigma-algebra $\mathcal{G}$, we define the $\mathbb P$-null sets in $\mathcal{G}$ as 
$$\mathcal{N}_{\mathbb P}(\mathcal{G}):=\{B\subset \Omega: B\subset C \text{ for some } C\in \mathcal{G} \text{ with } \mathbb P(C)=0\}.$$
When we extend sub sigma-algebras by adding null sets, with some abuse of notation, we will still denote the extended probability measure with the same notation.

\section{Discrete time optimal stopping MFGs with common noise and partial information}\label{abstract}

In this section, we develop the linear programming approach to solve discrete time optimal stopping MFGs with common noise  and partial information. In particular, we show the existence of an equilibria and provide a rigorous approximation result of the equilibria in the $N$-player game. We first present the results in the case of a single population and then we extend them to the case of several populations.

\subsection{Probabilistic set-up}\label{proba set}

Let $I=\{0, \ldots, T\}$ be the set of time indices with $I^*=\{1, \ldots, T\}$. We are given a nonempty compact metric space $(E, d)$ and a nonempty finite set $H$. \vspace{5pt}

\noindent Let $(\Omega, \mathcal F, \mathbb F, \mathbb P)$ be a complete filtered probability space supporting the process $(X, Z)$ satisfying:
\begin{enumerate}[(1)]
\item The state process $X=(X_t)_{t\in I}$ is an $\mathbb F$-adapted process taking values in $E$.
\item The common noise process $Z=(Z_t)_{t\in I}$ is an $\mathbb F$-adapted process taking values in $H$.
\end{enumerate}

Consider a nonempty compact set $D\subset E$ and the $\mathbb F$-stopping time $\tau^X_{D}:=\inf\{t\in I: X_t\notin D\}$ with the convention $\inf\emptyset=\infty$. The agents will undergo absorption if their state exits the set $D$. We assume that $X_0 \in D$ a.s.

\paragraph{Common noise information.} In the setting of MFGs with common noise, the mean-field terms will be conditional measures given the information of the common source of randomness which in our case is $Z$. Consider the associated filtration $\mathbb F^Z$, which models the common noise information conditioning the mean-field terms.\vspace{5pt}

The linear programming formulation without common noise for an optimal stopping problem consists in replacing the expectations of the processes and the stopping times by occupation measures and embedding these measures in a well behaved space (usually compact and convex) (see e.g. \cite{stockbridge2002, buckdahn2011stochastic, bdt2020, dlt2021}). This technique has been used in several works as \cite{stockbridge1998, stockbridge2002, buckdahn2011stochastic, bdt2020, dlt2021} in the case when the underlying processes are Markovian and the rewards depend only on the present states. However, in our setting, the reward functions will depend on the mean-field terms, which depend on the past of the common noise. The idea is then to construct another process summarizing the common noise information and such that together with the state processes of the representative agent, it defines a Markov process. This technique was used in e.g. \cite{kurtz1998} in the context of continuous time processes and filtered martingale problems.  \vspace{5pt}

The information given by the trajectory of the common noise $Z$ up to time $t$ will be summarized in a finite dimensional (matrix-valued) process $U_t$. Let $\mathcal{I}:=\{1, \ldots, |H|\}$ and write $H = \{z_j: j\in \mathcal{I}\}$ with $z_j\neq z_{j'}$ for $j\neq j'$, $(j, j')\in \mathcal{I}^2$. 
We define the process $U$ taking values in $W:=\{0, 1\}^{(N+1)\times |H|}$ (i.e. the space of $(N+1)\times |H|$-dimensional matrices with entries being $0$ or $1$) by
$$U_t(\omega)(s+1, j)=\mathds{1}_{s\leq t}\mathds{1}_{Z_{s}(\omega)=z_j},\quad (s, j)\in I\times \mathcal{I}, \quad t\in I,\quad \omega\in \Omega.$$
Alternatively, for $s\in I$ and $\bar z \in H$, let $M[s, \bar z]$ be the $(N+1)\times |H|$-dimensional matrix such that
$$M[s, \bar z](r+1, j)= \mathds{1}_{s=r}\mathds{1}_{\bar z = z_j}, \quad (r, j)\in I\times \mathcal{I}.$$
Then we can check that for $t\in I$, $U_t=\sum_{s=0}^t M[s, Z_s]$. For each $t\in I$ we define the function $\Psi_t:W\rightarrow H$ by
\begin{equation}\label{def psi}
\Psi_t(u):=\sum_{j\in \mathcal{I}}z_j\mathds{1}_{u(t+1, j)=1}, \quad u\in W.    
\end{equation}
This function satisfies $Z_s = \Psi_s(U_t)$ for all $s\leq t$ and $t\in I$, hence the information about the entire trajectory of $Z$ up to time $t$ is encoded in $U_t$. 
Moreover, we can conclude that $\sigma(U_t)=\mathcal{F}_t^{Z} =\mathcal{F}_t^{U}$, which also shows that $U$ is a Markov chain. We can also compute the transition kernels of $U$ as stated in the following lemma.
\begin{lemma}\label{U Markov}
The process $U$ is a Markov chain with transition kernels given for $t\in I^*$, $u$, $u'\in W$ by
\begin{align*}
\pi^U_t(u; u') :=
\begin{cases}
\sum_{z\in H}\mathds{1}_{u + M[t, z]=u'}\mathbb P(Z_t=z|U_{t-1}=u) \quad &\text{if} \quad  \mathbb P(U_{t-1}=u)>0, \\
\mathds{1}_{u'=u} \quad &\text{if} \quad  \mathbb P(U_{t-1}=u)=0.
\end{cases}
\end{align*}
\end{lemma}

\begin{proof}
For any $\varphi:W\rightarrow\mathbb R$ measurable and bounded,
\begin{align*}
\mathbb E[\varphi(U_t)|\mathcal{F}_{t-1}^U]=\mathbb E[\varphi(U_{t-1} + M[t, Z_t])|U_{t-1}] = \sum_{z\in H}\varphi(U_{t-1} + M[t, z])\mathbb P(Z_t=z|U_{t-1}).
\end{align*}
We deduce that a transition kernel for $U$ is given for $t\in I^*$, $u$, $u'\in W$ by
\begin{align*}
\pi^U_t(u; u') :=
\begin{cases}
\sum_{z\in H}\mathds{1}_{u + M[t, z]=u'}\mathbb P(Z_t=z|U_{t-1}=u) \quad &\text{if} \quad  \mathbb P(U_{t-1}=u)>0, \\
\mathds{1}_{u'=u} \quad &\text{if} \quad  \mathbb P(U_{t-1}=u)=0.
\end{cases}
\end{align*}

\end{proof}

To give the intuition behind the linear programming constraint that we will use to relax our problem, we start by defining for any stopping time the associated occupation measures and then we derive a forward equation they satisfy. {Assume for the rest of this subsection that $U$ and $(X_t, U_t)_{t\in I}$ are $\mathbb F$-Markov chains with transition kernels $(\pi_t^U)_{t\in I^*}$ and $(\pi_t)_{t\in I^*}$, respectively. Since $U$ is an $\mathbb F$-Markov chain, a similar proof to Lemma \ref{Immersion property} implies that $\mathbb F^U=\mathbb F^Z$ is immersed\footnote{Meaning that for all $t\in I$, $\mathcal{F}_t$ is conditionally independent of $\mathcal{F}^U_T$ given $\mathcal{F}^U_t$.} in $\mathbb F$. In particular, for any $\mathbb F$-adapted and integrable process $Y=(Y_t)_{t\in I}$, we have $\mathbb E[Y_t|\mathcal{F}^Z_t]=\mathbb E[Y_t|\mathcal{F}^Z_T]$ a.s. for all $t\in I$.}

\paragraph{Occupation measures.} For any $\mathbb F$-stopping time $\tau$ taking values in $I$ define the random subprobability measures $\tilde m_t:\Omega\rightarrow \mathcal{P}^{sub}(D)$ and $\tilde\mu_t:\Omega\rightarrow \mathcal{P}^{sub}(E)$ satisfying
\begin{align*}\tilde m_t(B)&:=\mathbb P[X_t\in B, t<\tau\wedge\tau_D^X|\mathcal{F}_t^Z],\quad a.s., \quad B\in \mathcal{B}(D),\quad t\in I\setminus\{T\},\\
\tilde \mu_t(B)&:=\mathbb P[X_t\in B, \tau\wedge\tau_D^X=t|\mathcal{F}_t^Z],\quad a.s., \quad B\in \mathcal{B}(E), \quad t\in I.
\end{align*}
Note that $\tilde m_t$ and $\tilde \mu_t$ are well defined as random variables since $E$ is a compact metric space and they can also be seen as subprobability kernels from $(\Omega, \mathcal{F})$ to $(E, \mathcal{B}(E))$. These random subprobability measures are the central objects of the MFG theory with optimal stopping and common noise. Here $\tilde m_t$ represents the distribution at time $t$ of the players still in the game. On the other hand, $\tilde \mu_t$ is the distribution at time $t$ of the players who exit the game at this time.

Moreover using that $\mathcal{F}_t^{Z}=\sigma(U_t)$ which is finitely generated by the sets $(\{U_t = u\})_{u}$, we can identify these random subprobabilities with the mappings $m_t:\Omega_t\rightarrow \mathcal{P}^{sub}(D)$ and $\mu_t:\Omega_t\rightarrow \mathcal{P}^{sub}(E)$ given by
\begin{align}
m_t(u)(B)&:=\mathbb P[X_t\in B, t<\tau\wedge\tau_D^X|U_t = u], \quad B\in \mathcal{B}(D),\quad t\in I\setminus\{T\}, \quad u\in \Omega_t,\label{occup1}\\
\mu_t(u)(B)&:=\mathbb P[X_t\in B, \tau\wedge\tau_D^X=t|U_t=u], \quad B\in \mathcal{B}(E), \quad t\in I,\quad u\in \Omega_t,\label{occup2}
\end{align}
where $\Omega_t:=\{u\in W:\mathbb P(U_t = u)>0\}$. We will rather work with the latter quantities which have the same interpretation as the first ones but at each particular trajectory of the common noise up to the corresponding time.\vspace{5pt}

\paragraph{Derivation of the constraint.} For any $\varphi\in M_b(I\times E\times W)$ we have
$$\mathbb E[\varphi(t+1, X_{t+1}, U_{t+1})- \varphi(t, X_{t}, U_t)|\mathcal{F}_t]=\mathcal{L}(\varphi)(t, X_{t}, U_t) \quad a.s,$$
where for $(t, x, u)\in I\setminus\{T\}\times E\times W$
$$\mathcal{L}(\varphi)(t, x, u):= \int_{E\times W} [\varphi(t+1, x', u')-\varphi(t, x, u)]\pi_{t+1}(x, u; dx', du').$$
Moreover the process $M(\varphi)$ defined by $M_0(\varphi)=\varphi(0, X_0, U_0)$ and
$$M_t(\varphi) = \varphi(t, X_t, U_t) - \sum_{s=0}^{t-1}\mathcal{L}(\varphi)(s, X_s, U_s), \quad t\in I^*,$$
is an $\mathbb F$-martingale. In particular for the stopping time $\theta:=\tau\wedge\tau_D^X$, we have
$\mathbb E[M_{\theta}(\varphi)|\mathcal{F}_0]=\varphi(0, X_0, U_0)$ a.s. Taking the expectation,
$$\mathbb E[\varphi(\theta, X_{\theta}, U_{\theta})]= \mathbb E[\varphi(0, X_0, U_0)] + \mathbb E\left[\sum_{t=0}^{\theta-1}\mathcal{L}(\varphi)(t, X_t, U_t)\right].$$
We deduce that the occupation measures satisfy the constraint

\begin{multline*}
\sum_{t=0}^T\sum_{u\in \Omega_t}p_t(u)\int_{E}\varphi(t, x, u) \mu_t(u)(dx) \\=\sum_{u\in \Omega_0}p_0(u)\int_{D} \varphi (0, x, u) m_0^*(u)(dx)
+ \sum_{t=0}^{T-1}\sum_{u\in \Omega_t}p_t(u)\int_{D}\mathcal{L}(\varphi)(t, x, u) m_t(u)(dx),
\end{multline*}
for any bounded and measurable test function $\varphi$, where $p_t(u):=\mathbb P(U_t = u)$ and $m_0^*(u)(B):=\mathbb P (X_0\in B|U_0 = u)$. In the sequel, we will restrict the constraint to continuous test functions (they are bounded by compactness), which are more suitable for our topological arguments. 

\subsection{Linear programming formulation}

We endow any discrete set with the discrete topology, the space $E$ is endowed with the topology induced by its metric and the set of subprobability measures on some Polish space is endowed with the topology of weak convergence. The topology on any product space is taken to be the product topology. \vspace{5pt}

\noindent We are given the following reward functions:
$$f_t:D\times W\times   \mathcal{P}^{sub}(D)\rightarrow \mathbb R,\quad t\in I\setminus \{T\},\quad \quad g_t:E\times W\times  \mathcal{P}^{sub}(E)\rightarrow \mathbb R, \quad t\in I.$$
In this subsection and the next one, we let the following assumptions hold true.

\begin{assumption}\label{assump existence}\leavevmode
\begin{enumerate}[(1)]
\item $X_0\in D$ a.s.

\item $(X, U)$ is an $\mathbb F$-Markov chain with transition kernels $(\pi_t)_{t\in I^*}$ such that for all $t\in I^*$, $x, \bar x\in E$, $u\in W$ and $u'\in W$, $\pi_t(x, u;E\times \{u'\})=\pi_t(\bar x, u;E\times \{u'\})$.

\item For each $t\in I^*$, $\pi_t$ is continuous seen as a function from $E\times W$ to $\mathcal{P}(E\times W)$.

\item For each $t$, the functions $(x, u, m)\mapsto f_t(x, u, m)$ and $(x, u, \mu)\mapsto g_t(x, u, \mu)$ are continuous.
\end{enumerate}
\end{assumption}

\begin{remark}
Note that the second assumption is equivalent to the statement that $(X, U)$ and $U$ are $\mathbb F$-Markov chains.
\end{remark}

We start by giving the definition of a strong equilibrium with strict stopping time. Here, we use the terminology \textit{strong} to represent flow of measures adapted to the filtration of the common noise and the terminology \textit{strict stopping time} refers to stopping times with respect to the filtration $\mathbb{F}$. Denote by $\mathcal{T}$ the set of $\mathbb F$-stopping times valued in $I$. 

\begin{definition}[\textit{Strong MFG equilibrium with strict stopping time}]\leavevmode
\begin{enumerate}[(1)]
\item Fix for each $t\in I$, $\mu_t:W\rightarrow \mathcal{P}^{sub}(E)$ and for each $t\in I\setminus\{T\}$, $m_t:W\rightarrow \mathcal{P}^{sub}(D)$ and find the solution to the optimal stopping problem
\begin{equation}\label{opti strong}
\sup_{\tau \in \mathcal{T}}\; \mathbb{E}\left[\sum_{t=0}^{\tau\wedge \tau_D^X-1} f_t\left(X_t, U_t, m_t(U_t)\right) + g_{\tau\wedge \tau_D^X}\left(X_{\tau\wedge \tau_D^X}, U_{\tau\wedge \tau_D^X}, \mu_{\tau\wedge \tau_D^X}(U_{\tau\wedge \tau_D^X})\right)\right].
\end{equation}
\item Denoting by $\tau^{\mu, m}$ an optimal stopping time associated to $(\mu, m)$ (solution of the problem \eqref{opti strong}), find $(\mu, m)$ such that
$$m_t(u)(B) = \mathbb P[X_t\in B, t<\tau^{\mu, m}\wedge\tau_D^X|U_t=u], \quad B\in\mathcal{B}(D), \quad t\in I\setminus\{T\}, \quad u\in \Omega_t,$$
$$\mu_t(u)(B)=\mathbb P[X_t\in B, \tau^{\mu, m}\wedge\tau_D^X=t|U_t=u], \quad B\in \mathcal{B}(E), \quad t\in I,\quad u\in \Omega_t.$$
\end{enumerate}
\end{definition}

\begin{remark}
One could take $f_t$ to depend on $(m_s)_{s\leq t}$ and $g_t$ to depend on $(\mu_s)_{s\leq t}$. In that case, the proofs will be the same as the ones that we will present by observing that for each $s\leq t$, $U_s$ is a deterministic function of $U_t$. In order to simplify the notations, we have decided not to do so.
\end{remark}

Now, if we are given $\bar \mu_t:W\rightarrow \mathcal{P}^{sub}(E)$, $\bar m_t:W\rightarrow \mathcal{P}^{sub}(D)$ and some $\tau\in \mathcal{T}$, we can define the occupation measures $(\mu, m)$ associated to $\tau$ as in the above subsection and rewrite the objective function:
\begin{multline*}
\mathbb{E}\left[\sum_{t=0}^{\tau\wedge \tau_D^X-1} f_t\left(X_t, U_t, \bar m_t(U_t)\right) + g_{\tau\wedge \tau_D^X}\left(X_{\tau\wedge \tau_D^X}, U_{\tau\wedge \tau_D^X}, \bar \mu_{\tau\wedge \tau_D^X}(U_{\tau\wedge \tau_D^X})\right)\right] \\
= \sum_{t=0}^{T-1}\sum_{u\in \Omega_t}p_t(u) \int_D f_t\left(x, u, \bar m_t(u)\right)m_t(u)(dx) + \sum_{t=0}^T\sum_{u\in \Omega_t}p_t(u)\int_E g_t(x, u, \bar \mu_t(u))\mu_t(u)(dx).
\end{multline*}
Having in mind this writing of the objective function and the derivation of the constraint for the occupation measures, we are going to present the definition of an equilibrium under the linear programming formulation.\vspace{5pt}

\noindent To ease the notation, we define the sets $\mathcal{C}_\mu := \prod_{t\in I}\mathcal{P}^{sub}(E)^{\Omega_t}$ and $\mathcal{C}_m := \prod_{t\in I\setminus\{T\}}\mathcal{P}^{sub}(D)^{\Omega_t}$,
where we use $\mathcal{Y}^\mathcal{X}$ to denote the set of functions from $\mathcal{X}$ to $\mathcal{Y}$. Since the sets $\Omega_t$ with $t\in I$ are finite, we identify these functions with vectors.

\begin{definition}[\textit{Set of constraints}]\label{Set of constraints}
Let $\mathcal{R}$ be the set of pairs $(\mu, m)\in \mathcal{C}_\mu\times \mathcal{C}_m$ such that for all $\varphi\in C(I\times E\times W)$,

\begin{multline*}
\sum_{t=0}^T\sum_{u\in \Omega_t}p_t(u)\int_{E}\varphi(t, x, u) \mu_t(u)(dx) \\=\sum_{u\in \Omega_0}p_0(u)\int_{D} \varphi (0, x, u) m_0^*(u)(dx)
+ \sum_{t=0}^{T-1}\sum_{u\in \Omega_t}p_t(u)\int_{D}\mathcal{L}(\varphi)(t, x, u) m_t(u)(dx).
\end{multline*}
\end{definition}


\begin{definition}[\textit{LP optimization criteria}]
For $(\bar \mu, \bar m)\in \mathcal{R}$, let $\Gamma[\bar \mu, \bar m]: \mathcal{R}\rightarrow \mathbb R$ be the reward functional associated to $(\bar \mu, \bar m)$, defined by
\begin{multline*}
\Gamma[\bar \mu, \bar m] (\mu, m)=\sum_{t=0}^{T-1}\sum_{u\in \Omega_t}p_t(u) \int_D f_t\left(x, u, \bar m_t(u)\right)m_t(u)(dx)
\\+ \sum_{t=0}^T\sum_{u\in \Omega_t}p_t(u)\int_E g_t(x, u, \bar \mu_t(u))\mu_t(u)(dx).
\end{multline*}
\end{definition}

\begin{definition}[\textit{LP MFG Nash equilibrium}]
We say that $(\mu^\star, m^\star)\in \mathcal{R}$ is an LP MFG Nash equilibrium if for all $(\mu, m)\in \mathcal{R}$, $\Gamma[\mu^\star, m^\star] (\mu, m)\leq \Gamma[\mu^\star, m^\star] (\mu^\star, m^\star)$.
\end{definition}

The set of measures $\mathcal R$ is larger than the set of occupation measures defined by \eqref{occup1}--\eqref{occup2}. However, the following theorem provides a probabilistic interpretation of the measures in $\mathcal{R}$ as \emph{occupation measures associated to randomized stopping times} in the sense of \cite{carmona2017}, i.e. stopping times not necessarily with respect to the filtration of the underlying Markov chains, $X$ and $U$, but with respect to larger filtrations for which the Markov property of $(X, U)$ can be extended.

\begin{theorem}[\textit{Probabilistic representation with randomized stopping times}]\label{proba rep U}
If $(\mu, m)\in \mathcal{R}$, then there exists $(\overline{\Omega}, \overline{\mathcal{F}}, \overline{\mathbb F}, \overline{\mathbb P}, \overline{\tau}, \overline{X}, \overline{U})$ verifying
\begin{enumerate}[(1)]
\item $(\overline{\Omega}, \overline{\mathcal{F}}, \overline{\mathbb P})$ is a complete probability space endowed with a complete filtration $\overline{\mathbb F}$ and supporting the random variables $(\overline{\tau}, \overline{X}, \overline{U})$.
\item $(\overline{X}, \overline{U})$ is an $\overline{\mathbb F}$-Markov chain valued in $E\times W$ with transition kernels $(\pi_t)_{t\in I^*}$ and initial distribution $m_0^*(u)(dx)p_0(u)$.
\item $\overline{\tau}$ is an $\overline{\mathbb F}$-stopping time valued in $I$ such that $\overline{\tau}\leq \tau_D^{\overline{X}}$.
\item The measures have the following representation:
$$m_t(u)(B)=\overline{\mathbb P}[\overline{X}_t\in B, t<\overline{\tau}|\overline{U}_t = u], \quad B\in \mathcal{B}(D),\quad t\in I\setminus\{T\}, \quad u\in \Omega_t,$$
$$\mu_t(u)(B):=\overline{\mathbb P}[\overline{X}_t\in B, \overline{\tau} = t|\overline{U}_t=u], \quad B\in \mathcal{B}(E), \quad t\in I,\quad u\in \Omega_t.$$
\end{enumerate}
\end{theorem}

\begin{proof}
For $B\in \mathcal{B}(E)$ and $u\in W$, define the following measures on $E$
$$\bar m_0(B\times \{u\}):=p_0(u)m_0^*(u)(B\cap D),$$
$$
\bar m_t(B\times \{u\}):=
\begin{cases}
p_t(u)m_t(u)(B\cap D), \quad &\text{if} \quad u\in \Omega_t\\
0, \quad &\text{if} \quad u\notin \Omega_t,
\end{cases}
\quad\quad\quad\quad t\in I\setminus\{T\},
$$
$$
\bar \mu_t(B\times \{u\}):=
\begin{cases}
p_t(u)\mu_t(u)(B), \quad &\text{if} \quad u\in \Omega_t\\
0, \quad &\text{if} \quad u\notin \Omega_t,
\end{cases}
\quad\quad\quad\quad t\in I.
$$
These measures satisfy for all $\varphi\in C(I\times E\times W)$,
\begin{multline*}
\sum_{t=0}^T\int_{E\times W}\varphi(t, x, u) \bar \mu_t(dx, du) \\=\int_{E\times W} \varphi(0, x,u)\bar m_0(dx, du) + \sum_{t=0}^{T-1}\int_{E\times W}\mathcal{L}(\varphi)(t, x, u) \bar m_t(dx, du).
\end{multline*}
By Theorem \ref{proba rep}, there exists $(\overline{\Omega}, \overline{\mathcal{F}}, \overline{\mathbb F}, \overline{\mathbb P}, \tilde \tau, \overline{X}, \overline{U})$ verifying
\begin{enumerate}[(1)]
\item $(\overline{\Omega}, \overline{\mathcal{F}}, \overline{\mathbb P})$ is a complete probability space endowed with a complete filtration $\overline{\mathbb F}$ and supporting the random variables $(\tilde\tau, \overline{X}, \overline{U})$.
\item $(\overline{X}, \overline{U})$ is an $\overline{\mathbb F}$-Markov chain valued in $E\times W$ with transition kernels $(\pi_t)_{t\in I^*}$ and initial distribution $\bar m_0$.
\item $\tilde \tau$ is an $\overline{\mathbb F}$-stopping time valued in $I$.
\item The measures have the following representation:
$$\bar m_t(B\times \{u\})=\overline{\mathbb P}[\overline{X}_t\in B, \overline{U}_t = u, t<\tilde \tau], \quad B\in \mathcal{B}(E), \quad u\in W, \quad t\in I\setminus\{T\},$$
$$\bar \mu_t(B\times \{u\}):=\overline{\mathbb P}[\overline{X}_t\in B, \overline{U}_t = u, \tilde \tau = t], \quad B\in \mathcal{B}(E), \quad u\in W, \quad t\in I.$$
\end{enumerate}
Now, we have that
\begin{align*}
\overline{\mathbb P}(\tau_D^{\overline{X}}<\tilde \tau)&=\overline{\mathbb P}(\cup_{t=0}^{T-1}\{\tau_D^{\overline{X}}\leq t<\tilde\tau\})\leq \sum_{t=0}^{T-1}\overline{\mathbb P}((\cup_{s=0}^t\{\overline{X}_s\in D^c\})\cap\{t<\tilde\tau\})\\
&\quad \leq \sum_{t=0}^{T-1}\sum_{s=0}^t\overline{\mathbb P}(\overline{X}_s\in D^c, t<\tilde\tau)\leq \sum_{t=0}^{T-1}\sum_{s=0}^t \bar m_s(D^c\times W)=0.
\end{align*}
This computation allows to define the $\overline{\mathbb F}$-stopping time $\overline{\tau}:=\tilde\tau\wedge\tau_D^{\overline{X}}$, which is equal to $\tilde \tau$ $\overline{\mathbb P}$-a.s. It is straightforward to show that
$$m_t(u)(B)=\overline{\mathbb P}[\overline{X}_t\in B, t<\overline{\tau}|\overline{U}_t = u], \quad B\in \mathcal{B}(D),\quad t\in I\setminus\{T\}, \quad u\in \Omega_t,$$
$$\mu_t(u)(B):=\overline{\mathbb P}[\overline{X}_t\in B, \overline{\tau} = t|\overline{U}_t=u], \quad B\in \mathcal{B}(E), \quad t\in I,\quad u\in \Omega_t.$$

\end{proof}

\subsection{Existence of LP MFG Nash equilibria}

In order to show the existence of an LP MFG Nash equilibrium we first show the compactness of the set of admissible measures and then we show that there exists a fixed point to the best response map via Kakutani-Fan-Glicksberg's fixed point theorem.

\paragraph{Topology.} We say that $(m^n)_{n\geq 1}$ converges to $m$ in $\mathcal{C}_m$ if for all $t\in I\setminus\{T\}$ and $u\in \Omega_t$, $(m^n_t(u))_{n\geq 1}$ converges to $m_t(u)$ weakly. We say that $(\mu^n)_{n\geq 1}$ converges to $\mu$ in $\mathcal{C}_\mu$ if for all $t\in I$ and $u\in \Omega_t$, $(\mu^n_t(u))_{n\geq 1}$ converges to $\mu_t(u)$ weakly. We recall that for a compact metric space $K$, the set $\mathcal{P}^{sub}(K)$ is compact for the topology of weak convergence.

\begin{theorem}\label{compact discrete}
The set $\mathcal{R}$ is compact.
\end{theorem}

\begin{proof}
Relative compactness follows since $\mathcal{P}^{sub}(E)$ and $\mathcal{P}^{sub}(D)$ are compact and henceforth $\mathcal{C}_\mu$ and $\mathcal{C}_m$ are also compact. Let us check that $\mathcal{R}$ is closed. Consider a sequence $(\mu^n, m^n)_{n\geq 1}\subset \mathcal{R}$ converging to some $(\mu, m)$ and let us show that $(\mu, m)\in \mathcal{R}$. For all $\varphi\in C(I\times E \times W)$ and $n\geq 1$,
\begin{multline*}
\sum_{t=0}^T\sum_{u\in \Omega_t}p_t(u)\int_{E}\varphi(t, x, u) \mu_t^n(u)(dx) \\=\sum_{u\in \Omega_0}p_0(u)\int_{D} \varphi (0, x, u) m_0^*(u)(dx)
+ \sum_{t=0}^{T-1}\sum_{u\in \Omega_t}p_t(u)\int_{D}\mathcal{L}(\varphi)(t, x, u) m_t^n(u)(dx).
\end{multline*}
Since for each $t\in I$ and $u\in \Omega_t$ the function $x\mapsto \varphi(t, x, u)$ is continuous and bounded we get that
$$\lim_{n\rightarrow\infty}\int_{E}\varphi(t, x, u) \mu_t^n(u)(dx) = \int_{E}\varphi(t, x, u) \mu_t(u)(dx).$$
Now using that for each $t\in I\setminus \{N\}$ and $u\in \Omega_t$ the function $x\mapsto \mathcal{L}(\varphi)(t, x, u)$ is continuous (by the continuity assumption on the transition kernel) and bounded, we have
$$\lim_{n\rightarrow\infty}\int_{D}\mathcal{L}(\varphi)(t, x, u) m_t^n(u)(dx) = \int_{D}\mathcal{L}(\varphi)(t, x, u) m_t(u)(dx).$$
We deduce that $(\mu, m)\in \mathcal{R}$.

\end{proof}

\noindent In order to find an equilibrium we define the following best response set valued mapping.

\begin{definition}
Define the set valued mapping $\Theta:\mathcal{R} \rightarrow 2^{\mathcal{R}}$ as 
$$\Theta(\bar\mu,\bar m)=\underset{(\mu,  m) \in \mathcal{R}}{\arg \max } \;\Gamma [\bar \mu, \bar m]( \mu,  m).$$
\end{definition}

\begin{remark}
Note that the set of LP MFG Nash equilibria coincides with the set of fixed points of $\Theta$.
\end{remark}

\noindent Before stating the existence theorem, we prove a preliminary lemma concerning the continuity of the reward functional.

\begin{lemma}\label{cont reward}
The map $\mathcal{R}^2\ni((\bar\mu, \bar m), (\mu, m))\mapsto \Gamma [\bar \mu, \bar m]( \mu,  m)\in \mathbb R$ is continuous.
\end{lemma}

\begin{proof}
Let $(\bar \mu^n, \bar m^n)_{n\geq 1}\subset \mathcal{R}$ and $(\mu^n, m^n)_{n\geq 1}\subset \mathcal{R}$ be two sequences converging to $(\bar \mu, \bar m)\in \mathcal{R}$ and $(\mu, m)\in \mathcal{R}$ respectively. Now, applying Lemma F.1 in \cite{dlt2021}, we get for all $t\in I\setminus \{T\}$ and $u\in \Omega_t$,
$$\int_{D}f_t(x, u, \bar m_t^n(u))m_t^n(u)(dx)\underset{n\rightarrow\infty}{\longrightarrow}\int_{D}f_t(x, u, \bar m_t(u))m_t(u)(dx),$$
and for all $t\in I$ and $u\in \Omega_t$,
$$\int_{E}g_t(x, u, \bar \mu^n_t(u))\mu_t^n(u)(dx)\underset{n\rightarrow\infty}{\longrightarrow}\int_{E}g_t(x, u, \bar \mu_t(u))\mu_t(u)(dx).$$
Since the reward functional is a linear combination of these quantities over $t$ and $u$, the limit is straightforward.

\end{proof}

\begin{theorem}\label{existence 1 player}
There exists an LP MFG Nash equilibrium.
\end{theorem}

\begin{proof}
We aim to apply Kakutani-Fan-Glicksberg's Theorem, Corollary 17.55 in \cite{aliprantis2007}. First note that $\mathcal{R}$ is nonempty, compact (Theorem \ref{compact discrete}), convex and is included in 
$$\prod_{t\in I}\mathcal{M}^s(E)^{\Omega_t}\times\prod_{t\in I\setminus\{T\}} \mathcal{M}^s(D)^{\Omega_t},$$
which is a locally convex Hausdorff space (endowed with the product topology and the set of finite signed measures endowed with the topology of weak convergence). Moreover, $\Theta$ has nonempty values since $\mathcal{R}$ is compact and for each $(\bar \mu,  \bar m) \in \mathcal{R}$, the map $\mathcal{R}\ni(\mu, m)\mapsto \Gamma [\bar \mu, \bar m]( \mu,  m)\in \mathbb R$ is continuous on $\mathcal{R}$ by Lemma \ref{cont reward}. This map is also linear which gives that $\Theta$ has convex values. Again, by Lemma \ref{cont reward}, the map $\mathcal{R}^2\ni((\bar\mu, \bar m), (\mu, m))\mapsto \Gamma [\bar \mu, \bar m]( \mu,  m)\in \mathbb R$ is continuous on $\mathcal{R}^2$, which shows that $\Theta$ has closed graph. By Kakutani-Fan-Glicksberg's fixed point theorem we conclude that there exists an equilibrium.

\end{proof}

\subsection{Existence of approximate equilibria in the $N$-player game}

In this subsection, we show that an LP MFG Nash equilibrium can be used to obtain $\varepsilon$-Nash equilibria for games with finite number of players.\vspace{5pt}

Let $(\Omega, \mathcal{F}, \mathbb P)$ be a complete probability space and let $(\mathbb F^N)_{N\geq 1}$ be a sequence of complete filtrations modeling the information available to the players in each $N$-player game. Consider players with state processes $X^n=(X^n_t)_{t\in I}$, $n\in \mathbb N^*$. Let $Z$ be the common noise and let $U$ be constructed as before.

\begin{assumption}\label{assump n player}\leavevmode
\begin{enumerate}[(1)]
\item For each $N\geq 1$ and $n\in \{1, \ldots, N\}$, $(X^n, U)$ is an $\mathbb F^N$-Markov chain taking values in $E\times W$ with initial distribution $m_0^*(u)(dx)p_0(u)$ and transition kernels $(\pi_t)_{t\in I^*}$.

\item The sequence of random variables $(X^n)_{n\geq 1}$ (valued in $E^I$) is i.i.d. given $U$ (valued in $W^I$) in the sense of Definition \ref{def cond iid}.
\end{enumerate}
\end{assumption}

We fix $N\geq 1$ and introduce the concepts for the $N$-player game. Let $\mathcal{K}_N$ denote the set of transition kernels $\kappa$ from $\Omega$ to $I$ such that for all $t\in I$ and $B\in \sigma(\{0\}, \ldots, \{t\})$ the mapping $\omega\mapsto \kappa(\omega, B)$ is $\mathcal{F}_t^N$-measurable. We say that $\mathcal{K}_N$ is the set of randomized stopping times. A special class of randomized stopping times are the Markovian randomized stopping times (i.e. the randomness in the kernel comes from observing the underlying state process and common noise process). More precisely, denote by $\Omega_{X, U}:=(E\times W)^I$ the canonical space for the process $(X, U)$, which is endowed with the Borel $\sigma$-algebra, and define the filtration $\mathcal{F}_t^{X, U}:= \sigma(X_s, U_s: s\leq t)$, for $t\in I$. Let $\nu\in \mathcal{P}(\Omega_{X, U})$ be the unique law of the Markov chain with transition kernels $(\pi_t)_{t\in I^*}$ and initial law $m_0^*(u)(dx)p_0(u)$. We denote by $\mathbb F^{X, U, \nu}=(\mathcal{F}_t^{X, U, \nu})_{t\in I}$ the filtration $\mathcal{F}_t^{X, U, \nu}:=\mathcal{F}_t^{X, U}\vee \mathcal{N}_{\nu}(\mathcal{F}_T^{X, U})$. Define $\mathcal{K}^{X, U}$ as the set of transition kernels $\kappa$ from $\Omega_{X, U}$ to $I$ such that for all $t\in I$ and $B\in \sigma(\{0\}, \ldots, \{t\})$, the mapping $(x, u)\mapsto \kappa(x, u; B)$ is $\mathcal{F}_t^{X, U, \nu}$-measurable. 

\begin{lemma}
For all $n\in \{1, \ldots, N\}$ and all $\kappa\in \mathcal{K}^{X, U}$, $\kappa(X^n(\cdot), U(\cdot); \cdot)\in \mathcal{K}_N$. 
\end{lemma}

\begin{proof}
It suffices to show that for all $n\in \{1, \ldots, N\}$, $(X^n, U)$ is $\mathcal{F}_t^N/\mathcal{F}_t^{X, U, \nu}$-measurable for each $t\in I$. Take $B\in \mathcal{F}_t^{X, U, \nu}$. We can write $B=C\cup D$ with $C\in\mathcal{F}_t^{X, U}$ and $D\in \mathcal{N}_{\nu}(\mathcal{F}_T^{X, U})$. Now, we have that $(X^n, U)^{-1}(B)=(X^n, U)^{-1}(C)\cup (X^n, U)^{-1}(D)\in \mathcal{F}_t^N$. In fact, $(X^n, U)^{-1}(C)\in \mathcal{F}_t^N$ using that $(X^n, U)$ is $\mathbb F^N$-adapted, and $(X^n, U)^{-1}(D)\in \mathcal{N}_{\mathbb P}(\mathcal{F})$.

\end{proof}

We define the set of admissible strategies for the $N$-player game as the product set $\mathcal{K}_N^N$.
For an admissible strategy $\bm{\kappa}:=(\kappa^1, \ldots, \kappa^N)\in \mathcal{K}_N^N$, define the empirical measures
$$m^N_t[\bm{\kappa}](B)=\frac{1}{N}\sum_{n=1}^N\delta_{X^n_t}(B)\mathds{1}_{t<\tau_D^{X^n}}\kappa^n(\{t+1, \ldots, T\}), \quad B\in \mathcal{B}(D),\quad t\in I\setminus\{T\},$$
$$\mu^N_t[\bm{\kappa}](B)=\frac{1}{N}\sum_{n=1}^N\delta_{X^n_t}(B)\mathds{1}_{\tau_D^{X^n}=t}\kappa^n(\{t\}), \quad B\in \mathcal{B}(E),\quad t\in I.$$
Moreover, define the expected gain for player $n\in \{1, \ldots, N\}$ as 
\begin{multline*}J_n^N(\bm{\kappa}):=\mathbb E\Big[\sum_{t=0}^{T\wedge \tau_D^{X^n}-1}f_t(X^n_t, U_t, m^N_t[\bm{\kappa}])\kappa^n(\{t+1, \ldots, T\}) \\+ \sum_{t=0}^{T\wedge \tau_D^{X^n}}g_t(X^n_t, U_t, \mu^N_t[\bm{\kappa}])\kappa^n(\{t\})\Big].\end{multline*}
\begin{remark}
Let $\tau^1, \ldots, \tau^N$ be $\mathbb F^N$-stopping times and define $\kappa^n=\delta_{\tau^n}$. Then $\bm{\kappa}:=(\kappa^1, \ldots, \kappa^N)\in \mathcal{K}_N^N$, and we have:
$$m^N_t[\bm{\kappa}](B)=\frac{1}{N}\sum_{n=1}^N\delta_{X^n_t}(B)\mathds{1}_{t<\tau^n\wedge \tau_D^{X^n}}, \quad \mu^N_t[\bm{\kappa}](B)=\frac{1}{N}\sum_{n=1}^N\delta_{X^n_t}(B)\mathds{1}_{\tau^n\wedge \tau_D^{X^n}=t},$$
and,
$$J_n^N(\bm{\kappa}):=\mathbb E\left[\sum_{t=0}^{\tau^n\wedge \tau_D^{X^n}-1}f_t(X^n_t, U_t, m^N_t[\bm{\kappa}]) + g_{\tau^n\wedge \tau_D^{X^n}}\left(X^n_{\tau^n \wedge \tau_D^{X^n}},  U_{\tau^n \wedge \tau_D^{X^n}}, \mu^N_{\tau^n \wedge \tau_D^{X^n}}[\bm{\kappa}]\right)\right].$$
\end{remark}

\begin{definition}
Given $N\in \mathbb N^*$ and $\varepsilon\geq 0$, we say that $\bm{\kappa}:=(\kappa^1, \ldots, \kappa^N)\in \mathcal{K}_N^N$ is an $\varepsilon$-Nash equilibrium for the $N$-player game if $J^N_n(\bar\kappa, \bm{\kappa}^{-n})-\varepsilon\leq J^N_n(\bm{\kappa})$, for all $\bar\kappa\in \mathcal{K}_N$ and $n=1, \ldots, N$.
\end{definition}

Let $(\mu^\star, m^\star)$ be an LP MFG Nash equilibrium, which exists by Theorem \ref{existence 1 player}. Let $(\overline{\Omega}, \overline{\mathcal{F}}, \overline{\mathbb F}, \overline{\mathbb P}, \overline{\tau}, \overline{X}, \overline{U})$ be the associated probabilistic representation given by Theorem \ref{proba rep U}. The probability measure $P^\star:= \overline{\mathbb P}\circ (\overline{X}, \overline{U}, \overline{\tau})^{-1}\in \mathcal{P}(\Omega_{X, U}\times I)$ disintegrates as $P^\star(dx, du, d\theta)=\nu(dx, du)\kappa^\star(x, u;d\theta)$ for some $\kappa^\star\in \mathcal{K}^{X, U}$ (see e.g. Corollary \ref{coro disint}). Now define the Markovian randomized stopping times as $\kappa^n:=\kappa^\star(X^n(\cdot), U(\cdot); \cdot)$, for $n\geq 1$.\vspace{5pt}

\begin{theorem}
Let Assumptions \ref{assump existence} and \ref{assump n player} hold true. For any $\varepsilon>0$, there is some $N_0\geq 1$ such that if $N\geq N_0$, then $\bm{\kappa}=(\kappa^1, \ldots, \kappa^N)$ is an $\varepsilon$-Nash equilibrium for the $N$-player game.
\end{theorem}

\begin{proof}
Let $\varepsilon>0$. By the symmetry of the problem, it suffices to show that for $N$ large enough, 
\begin{equation}\label{conv N player}
\sup_{\bar\kappa \in \mathcal{K}_N} J^N_1(\bar\kappa, \bm{\kappa}^{-1})-\varepsilon\leq J^N_1(\bm{\kappa}).
\end{equation}
In fact, if $n\leq N$, since $(X^n, U)_{n\geq 1}$ is exchangeable, we obtain that $J_n^N(\bm{\kappa})=J_1^N(\bm{\kappa})$. Let us show that
$$\sup_{\bar \kappa\in \mathcal{K}_N}J_n^N(\bar\kappa, \bm{\kappa}^{-n})=\sup_{\bar \kappa\in \mathcal{K}_N}J_1^N(\bar\kappa, \bm{\kappa}^{-1}).$$
We observe that the supremum over randomized stopping times with respect to $\mathbb F^N$ is equal to the supremum over strict stopping times with respect to $\mathbb F^N$ (see Proposition 1.5 in \cite{edgar1982}) and the latter supremum is equal to the supremum over strict stopping times with respect to the complete filtration of $X^1, \ldots, X^N$ and $U$. Now, any strict stopping times with respect to the complete filtration of $X^1, \ldots, X^N$ and $U$ can be represented a.s. by a function of the underlying processes such that under permutation of the $(X^n)_{n\leq N}$ it is also a strict stopping time with respect to the same filtration. Again using that $(X^n, U)_{n\geq 1}$ is exchangeable, we obtain the equality.\vspace{5pt}

\noindent We show now result \eqref{conv N player}. 

\textit{Step 1.} To do so, we first show that
\begin{align}\label{convergence N player 1}
J_1^N(\bm{\kappa})\rightarrow \Gamma[\mu^\star, m^\star](\mu^\star, m^\star).
\end{align}
Let $t\in I\setminus\{T\}$ and $\varphi\in C(E)$. By Theorem \ref{clln},
\begin{align*}
\int_E\varphi(x)m^N_t[\bm{\kappa}](dx) &=\frac{1}{N}\sum_{n=1}^N\varphi(X^n_t)\mathds{1}_{t<\tau_D^{X^n}}\kappa^\star(X^n, U; \{t+1, \ldots, T\})\\
&\underset{N\rightarrow\infty}{\longrightarrow} \mathbb E[\varphi(X^1_t)\mathds{1}_{t<\tau_D^{X^1}}\kappa^\star(X^1, U; \{t+1, \ldots, T\})|U]\\
&=\int_E\varphi(x)m_t^\star(U_t)(dx).
\end{align*}
Taking $\varphi$ from a countable family characterizing the weak convergence (i.e. a sequence $(\varphi_k)_k\subset C(E)$ such that weak convergence of probability measures $\nu^n \rightharpoonup \nu$ is equivalent to $\int \varphi_k d\nu^n \rightarrow \int \varphi_k d\nu$ for all $k$), we get that $m^N_t[\bm{\kappa}]\rightharpoonup m_t^\star(U_t)$ a.s.
The existence of such a countable family characterizing the weak convergence holds by the separability of $C(E)$, but in more general cases one can also use  Theorem 6.6 in Chapter 2 in \cite{parthasarathy2005probability}.
Analogously, $\mu^N_t[\bm{\kappa}]\rightharpoonup \mu_t^\star(U_t)$ a.s. By dominated convergence we get that $\eqref{convergence N player 1}$ holds.

\noindent \textit{Step 2.} Let us show now that 
\begin{equation}\label{convergence N player 2}
\lim_{N\rightarrow\infty}\sup_{\bar\kappa\in \mathcal{K}_N} |J_1^N(\bar\kappa, \bm{\kappa}^{-1}) - \Gamma[\mu^\star, m^\star](\mu^{\bar \kappa}, m^{\bar \kappa})| = 0,
\end{equation}
where $(\mu^{\bar \kappa}, m^{\bar \kappa})$ are the measures associated to $\bar \kappa$:
$$m_t^{\bar \kappa}(u)(B):=\mathbb E[\mathds 1_B(X_t^1)\mathds 1_{t<\tau_D^{X^1}}\bar\kappa (\{t+1, \ldots, T\})|U_t = u],$$
$$\mu_t^{\bar \kappa}(u)(B):=\mathbb E[\mathds 1_B(X_t^1)\mathds 1_{t<\tau_D^{X^1}}\bar\kappa (\{t\})|U_t = u].$$
Note that by Proposition \ref{Dynkin rand}, $(\mu^{\bar \kappa}, m^{\bar \kappa})\in \mathcal{R}$. We will show the convergence only for the ``$m$ part'' since the ``$\mu$ part'' is analogous. \vspace{5pt}

\noindent First we will prove the following result: 
\begin{equation}\label{first result uniform}
\mathbb P\left(\lim_{N\rightarrow\infty}\sup_{\bar \kappa\in \mathcal{K}_N} d_{\text{BL}}(m^N_t[(\bar\kappa, \bm{\kappa}^{-1})], m_t^\star(U_t))=0\right)=1,
\end{equation}
where $d_{\text{BL}}$ is the bounded Lipschitz distance (see \cite{bogachev2007}, volume II, p.192 and Theorem 8.3.2 p.193), which metrizes the topology of weak convergence.

Let $(\varphi_k)_{k\geq 1}\subset C(E)$ be a sequence of functions characterizing weak convergence. By Theorem \ref{clln} we can find $A_k\in \mathcal{F}$ with $\mathbb P(A_k)=1$ and such that on the set $A_k$,
\begin{multline*}
\frac{1}{N}\sum_{n=2}^N\varphi_k(X^n_t)\mathds{1}_{t<\tau_D^{X^n}}\kappa^\star(X^n, U; \{t+1, \ldots, T\}) \underset{N\rightarrow\infty}{\longrightarrow} \\\mathbb E[\varphi_k(X^1_t)\mathds{1}_{t<\tau_D^{X^1}}\kappa^\star(X^1, U; \{t+1, \ldots, T\})|U] =\int_E\varphi_k(x)m_t^\star(U_t)(dx).
\end{multline*}
If we define $A:=\cap_k A_k\in\mathcal{F}$, we have that $\mathbb P(A)=1$ and on the set $A$,
$$\lim_{N\rightarrow \infty}d_{\text{BL}}(\tilde m^N_t[\bm{\kappa}], m_t^\star(U_t)) = 0,$$
where
$$\tilde m^N_t[\bm{\kappa}](B):= \frac{1}{N}\sum_{n=2}^N\delta_{X^n_t}(B)\mathds{1}_{t<\tau_D^{X^n}}\kappa^\star(X^n, U; \{t+1, \ldots, T\}).$$
We deduce that on the set $A$,
$$\sup_{\bar \kappa\in \mathcal{K}_N} d_{\text{BL}}(m^N_t[(\bar\kappa, \bm{\kappa}^{-1})], m_t^\star(U_t))\leq d_{\text{BL}}(\tilde m^N_t[\bm{\kappa}], m_t^\star(U_t)) + \frac{1}{N}\underset{N\rightarrow\infty}{\longrightarrow} 0,$$
which proves \eqref{first result uniform}.\vspace{5pt}

\noindent Since $(t, x, u, m)\mapsto f(t, x, u, m)$ is continuous on a compact set, henceforth uniformly continuous, we can find a bounded modulus of continuity $\omega_f$ associated to $f$ and the bounded Lipschitz distance. In particular,
\begin{align*}
&\sup_{\bar \kappa\in \mathcal{K}_N} \left|\mathbb E\left[\sum_{t=0}^{T\wedge \tau_D^{X^1}-1}f_t(X^1_t, U_t, m^N_t[(\bar\kappa, \bm{\kappa}^{-1})])\bar\kappa(\{t+1, \ldots, T\})\right]\right.\\
&\quad - \left.\mathbb E\left[\sum_{t=0}^{T\wedge \tau_D^{X^1}-1}f_t(X^1_t, U_t, m_t^\star(U_t))\bar\kappa(\{t+1, \ldots, T\})\right]\right|\\
&\leq \sum_{t=0}^{T-1}\sup_{\bar \kappa\in \mathcal{K}_N} \mathbb E[\omega_f(d_{\text{BL}}(m^N_t[(\bar\kappa, \bm{\kappa}^{-1})], m_t^\star(U_t)))]. 
\end{align*}
Now,  by definition of the supremum, there exists a sequence,  depending on $t$, $N$, $\bm{\kappa}^{-1}$, and $m_\cdot^\star(U_\cdot)$, that we denote by $(\bar \kappa_k^{(t, N)})_{k\geq 1}$, such that  $(\bar \kappa_k^{(t, N)})_{k\geq 1}\subset \mathcal{K}_N$ and
$$\sup_{\bar \kappa\in \mathcal{K}_N} \mathbb E[\omega_f(d_{\text{BL}}(m^N_t[(\bar\kappa, \bm{\kappa}^{-1})], m_t^\star(U_t)))] = \sup_{k\geq 1} \mathbb E[\omega_f(d_{\text{BL}}(m^N_t[(\bar \kappa_k^{(t, N)}, \bm{\kappa}^{-1})], m_t^\star(U_t)))].$$
Finally,
\begin{multline*}
\sup_{k\geq 1} \mathbb E[\omega_f(d_{\text{BL}}(m^N_t[(\bar \kappa_k^{(t, N)}, \bm{\kappa}^{-1})], m_t^\star(U_t)))] \leq \\
\mathbb E\left[\omega_f\left(\sup_{k\geq 1} d_{\text{BL}}(m^N_t[(\bar \kappa_k^{(t, N)}, \bm{\kappa}^{-1})], m_t^\star(U_t))\right)\right],
\end{multline*}
where we used that the modulus is nondecreasing. Since the supremum on the right-hand side is measurable and using the convergence result in \eqref{first result uniform}, by dominated convergence, the latter quantity converges to $0$ as $N\rightarrow\infty$. \vspace{5pt}

To conclude, we use \eqref{convergence N player 1} and \eqref{convergence N player 2} to find $N_0$ big enough so that for all $N\geq N_0$ and $\bar\kappa\in \mathcal{K}_N$
$$J_1^N(\bm{\kappa})\geq \Gamma[\mu^\star, m^\star](\mu^\star, m^\star) -\frac{\varepsilon}{2}, \quad \Gamma[\mu^\star, m^\star](\mu^{\bar\kappa}, \bar m^{\bar\kappa})\geq  J_1^N(\bar\kappa, \bm{\kappa}^{-1})-\frac{\varepsilon}{2},$$
leading to
$$J^N_1(\bar\kappa, \bm{\kappa}^{-1})-\varepsilon\leq J^N_1(\bm{\kappa}).$$

\end{proof}

\subsection{Scenario uncertainty}
To fix the ideas, consider a complete probability space $(\Omega, \mathcal{F}, \mathbb P)$ supporting $(X, S, Z)$ satisfying:
\begin{enumerate}[(1)] 
\item The process $X=(X_t)_{t\in I}$ is a Markov chain taking values in $E$ with transition kernels $(\pi_t^X)_{t\in I^*}$ and with initial law $m_0^*$.
\item The random variable $S$ takes values in $\mathcal{S}:=\{1, \ldots, n_S\}$, $n_S\geq 1$.
\item The process $(S, Z_t)_{t\in I}$ is a Markov chain taking values in $\mathcal{S}\times H$, with initial distribution $\Pi_0\times \delta_{z_0}$, where $z_0\in H$ is deterministic, and such that
$$\mathbb P(Z_t = z'|\mathcal{F}_{t-1}^{S, Z})=\pi_t^Z(S, Z_{t-1}; z') \quad a.s.,\quad z'\in H,\; t\in I^*,$$
for some transition kernels $(\pi_t^Z)_{t\in I^*}$ from $\mathcal{S}\times H$ to $H$.
\item The random elements $X$ and $(S, Z)$ are independent.
\end{enumerate}

\begin{remark}
Our probabilistic set-up can include more general cases of scenario uncertainty in which $(S, Z_t)_{t\in I}$ is not Markovian and $X$ can depend on the past of $Z$.
\end{remark}

The random variable $S$ represents the random scenario, which is not observable. However the \textit{prior distribution} $\Pi_0:=\mathcal{L}(S)$ is known. In stochastic control and optimal stopping problems with partial information, we usually want to maximize some expected reward depending on $X$, $Z$ and $S$ in a Markovian way. In the case of partial information, when we do not observe the scenario $S$, the process $Z$ is usually not Markovian by itself. In that case, the problem can be reformulated by replacing $S_t$ with the conditional law of $S_t$ given $\mathcal{F}^{Z}_t$, which is also called the posterior distribution, and obtain a new Markovian setting (e.g. \cite{rieder1975, Pham2005}). However in our framework it will not be necessary to introduce this conditional law process into the state process since we will add the process $U$ presented before, which encompasses it (since it fully characterizes the filtration of $Z$). \vspace{5pt}

\noindent Consider the filtration $\mathbb F:=\mathbb F^{X, U}=\mathbb F^{X, Z}$, then $(X, U)$ is an $\mathbb F$-Markov chain with transition kernels $(\pi_t)_{t\in I^*}$ given by
$$\pi_t((x, u), B\times \{u'\}):=\pi_t^X(x; B)\pi_t^U(u;u'), \quad x\in E, B\in \mathcal{B}(E), u, u'\in W.$$
In particular, if the corresponding assumptions on the initial distribution and the transition kernels are satisfied, we can apply Theorem \ref{existence 1 player}.\vspace{5pt}

Using Lemma \ref{U Markov}, we obtain that $U$ is a Markov chain with transition kernels given for $t\in I$, $u$, $u'\in W$ by
\begin{align*}
\pi^U_t(u; u') :=
\begin{cases}
\sum_{z\in H}\mathds{1}_{u + M[t, z]=u'}\mathbb P(Z_t=z|U_{t-1}=u) \quad &\text{if} \quad  \mathbb P(U_{t-1}=u)>0, \\
\mathds{1}_{u'=u} \quad &\text{if} \quad  \mathbb P(U_{t-1}=u)=0.
\end{cases}
\end{align*}
We observe that $\mathbb P\left(Z_t=z|U_{t-1}=u\right)=\mathbb P\left(Z_t=z |Z_1 = \Psi_1(u), \ldots, Z_{t-1} = \Psi_{t-1}(u)\right)$. Moreover, when $\mathbb P(Z_1=z_1, \ldots, Z_t=z_t)>0$, setting $z_0 = 0$, we have
\begin{align*}
\mathbb P(Z_1=z_1, \ldots, Z_t=z_t) &= \sum_{s\in \mathcal{S}} \mathbb P(Z_0=z_0, \ldots, Z_t=z_t|S = s)\mathbb P(S=s)\\
&= \sum_{s\in \mathcal{S}} \left[\prod_{k=1}^t\mathbb P(Z_k=z_{k}|Z_{k-1}=z_{k-1}, S=s)\right]\mathbb P(S=s)\\
&= \sum_{s\in \mathcal{S}} \left[\prod_{k=1}^t\pi^Z_t(s, z_{k-1};z_k)\right]\Pi_0(s),
\end{align*}
which allows to fully determine $\mathbb P(Z_t=z|U_{t-1}=u)$.

\subsection{Extension to several populations}

We extend the above results to the case of several populations which will be the situation in our application. We assume that there is a total number of $K\in \mathbb N^*$ different populations of agents.\vspace{5pt}


\noindent We are given $K$ compact metric spaces $(E_k, d_k)$, $k\in \{1, \ldots, K\}$, and compact subsets $D_k\subset E_k$ for each $k\in\{1, \ldots, K\}$.\vspace{5pt}

\noindent Let $(\Omega, \mathcal F, \mathbb P)$ be a complete probability space together with a complete filtration $\mathbb F$ and supporting $(\bm{X}, Z)$ where $\bm{X}:=(X^1, \ldots, X^K)$ satisfying:
\begin{enumerate}[(1)]
\item For $k\in\{1, \ldots, K\}$, the state process $X^k=(X_t^k)_{t\in I}$ is an $\mathbb F$-adapted process taking values in $E_k$.
\item The common noise process $Z=(Z_t)_{t\in I}$ is an $\mathbb F$-adapted process taking values in $H$.
\end{enumerate}

\noindent Let $U$, $W$ and $(\Omega_t)_{t\in I}$ be defined as in the previous subsection. We are given the following reward functions for $k\in\{1, \ldots, K\}$ and $t\in I$:
$$f^k_t:D_k\times W\times \prod_{i=1}^K\mathcal{P}^{sub}(D_i)\rightarrow \mathbb R, \quad t\in I\setminus\{T\},\quad \quad g_t^k:E_k\times W\times  \prod_{i=1}^K \mathcal{P}^{sub}(E_i)\rightarrow \mathbb R.$$
In this subsection, we let the following assumptions hold true.

\begin{assumption}\label{assump existence K pop}\leavevmode
\begin{enumerate}[(1)]
\item $X_0^k\in D_k$ a.s. for $k\in\{1, \ldots, K\}$.

\item For each $k\in\{1, \ldots, K\}$, $(X^k, U)$ is an $\mathbb F$-Markov chain with transition kernels $(\pi_t^k)_{t\in I^*}$ such that for all $t\in I^*$, $x, \bar x\in E_k$, $u\in W$ and $u'\in W$, $\pi_t^k(x, u;E_k\times \{u'\})=\pi_t^k(\bar x, u;E_k\times \{u'\})$.

\item For each $t\in I^*$, and $k\in\{1, \ldots, K\}$, $\pi_t^{k}$ is continuous seen as a function from $E_k\times W$ to $\mathcal{P}(E_k\times W)$.

\item For each $t$ and $k\in\{1, \ldots, K\}$, the functions $(x, u, \bm{m})\mapsto f_t^k(x, u, \bm{m})$ and $(x, u, \bm{\mu})\mapsto g_t^k(x, u, \bm{\mu})$ are continuous.
\end{enumerate}
\end{assumption}

\noindent For $k\in\{1, \ldots, K\}$, denote by $\mathcal{L}_k$ the operator
$$\mathcal{L}_k(\varphi)(t, x, u):= \int_{E_k\times W} [\varphi(t+1, x', u')-\varphi(t, x, u)]\pi_{t+1}^k(x, u, dx', du'),$$
where $\varphi\in M_b(I\times E_k\times W)$ and $(t, x, u)\in I\setminus\{T\}\times E_k\times W$. We define also $m_0^{*, k}(u)(B):= \mathbb P(X_0^k\in B|U_0 = u)$. As in the previous section we define the sets $\mathcal{C}_\mu^k := \prod_{t\in I}\mathcal{P}^{sub}(E_k)^{\Omega_t}$ and $\mathcal{C}_m^k := \prod_{t\in I\setminus\{T\}}\mathcal{P}^{sub}(D_k)^{\Omega_t}$, for $k\in \{1, \ldots, K\}$.

\begin{definition}
For $k\in\{1, \ldots, K\}$, let $\mathcal{R}_k$ be the set of pairs $(\mu, m)\in \mathcal{C}_\mu^k\times \mathcal{C}_m^k$ such that for all $\varphi\in C(I\times E\times W)$,
\begin{multline*}
\sum_{t=0}^T\sum_{u\in \Omega_t}p_t(u)\int_{E_k}\varphi(t, x, u) \mu_t(u)(dx) =\sum_{u\in \Omega_0}p_0(u)\int_{D_k} \varphi(0, x, u)m_0^{*, k}(u)(dx)\\
+ \sum_{t=0}^{T-1}\sum_{u\in \Omega_t}p_t(u)\int_{D_k}\mathcal{L}_k(\varphi)(t, x, u) m_t(u)(dx).
\end{multline*}
\end{definition}

\noindent With some abuse of notation we will write $(\bm{\mu}, \bm{m})\in \prod_{k=1}^K\mathcal{R}_k$ instead of $((\mu^1, m^1), \ldots, (\mu^K, m^K))\in \prod_{k=1}^K\mathcal{R}_k$.

\begin{definition}
For $(\bar{\bm{\mu}}, \bar{\bm{m}})\in \prod_{k=1}^K\mathcal{R}_k$ and $k\in\{1, \ldots, K\}$, let $\Gamma_k[\bar{\bm{\mu}}, \bar{\bm{m}}]: \mathcal{R}_k\rightarrow \mathbb R$ be the reward functional of the population $k$ associated to $(\bar{\bm{\mu}}, \bar{\bm{m}})$ and defined by
\begin{multline*}
\Gamma_k[\bar{\bm{\mu}}, \bar{\bm{m}}] (\mu, m)=\sum_{t=0}^{T-1}\sum_{u\in \Omega_t}p_t(u) \int_{D_k} f_t^k\left(x, u,\bar{\bm{m}}_t(u)\right)m_t(u)(dx) \\
+ \sum_{t=0}^T\sum_{u\in \Omega_t}p_t(u)\int_{E_k} g_t^k(x, u, \bar{\bm{\mu}}_t(u))\mu_t(u)(dx).
\end{multline*}
\end{definition}

\begin{definition}
We say that $(\bm{\mu}^\star, \bm{m}^\star)\in \prod_{k=1}^K\mathcal{R}_k$ is an LP MFG Nash equilibrium if for all $k\in\{1, \ldots, K\}$ and all $(\mu, m)\in \mathcal{R}_k$, 
$$\Gamma_k[\bm{\mu}^\star, \bm{m}^\star](\mu, m)\leq \Gamma_k[\bm{\mu}^\star, \bm{m}^\star] (\mu^{k, \star}, m^{k, \star}).$$
\end{definition}

Consider the best response map
$\Theta:\prod_{k=1}^K\mathcal{R}_k \rightarrow 2^{\prod_{k=1}^K\mathcal{R}_k}$ defined by 
$$\Theta(\bar{\bm{\mu}}, \bar{\bm{m}})=\prod_{k=1}^K\underset{(\mu,  m) \in \mathcal{R}_k}{\arg \max } \;\Gamma_k[\bar{\bm{\mu}}, \bar{\bm{m}}] (\mu, m).$$
Using a similar proof to Theorem \ref{existence 1 player}, but using the new best response map $\Theta$ we obtain the following existence result.

\begin{theorem}\label{nash K}
There exists an LP MFG Nash equilibrium.
\end{theorem}

\section{Application to energy transition under scenario uncertainty}
In this section we extend the model of \cite{adt2021} to account for transition scenario uncertainty. The aim of this stylized model is to simulate the electricity market dynamics (arrivals and departures of producers) over a long time period (e.g., 20 years), accounting for individual optimization objectives of the producers, in the presence of exogenous electricity demand. In the model there are three types of producers: the conventional producers (e.g., gas-fired power plants), who aim to exit the market at the optimal time, the potential renewable projects (e.g., wind power plants) who aim to enter the market at the optimal time, and the baseline producers (e.g., nuclear), who are present in the market during the entire simulation. The conventional producers determine the amount of electricity they produce but have random production costs and the renewable producers have zero production cost but produce a random amount of electricity per unit time. The demand is deterministic, and it is assumed that there is a large number of producers of both types, and that their production costs of conventional plants and the capacity factors of renewable plants are all independent. As a result, in the MFG limit, the randomness disappears and the electricity price becomes deterministic. 

The assumption of deterministic electricity demand and independent cost and capacity factor processes is not very realistic: in practice, electricity demand is affected by economic situation, production costs depend on fuel prices and capacity factors are determined by weather patterns, which may affect different producers in a similar way. All this calls for the inclusion of common noise in the model, which would not disappear in the limit and also affect the electricity prices. Indeed, although the model does not focus on short-term electricity prices but rather on longer-term averages, these are still random and depend on all the above factors. 

In this section we therefore extend this model based on the theory of optimal stopping games with common noise, developed in Section \ref{abstract}. Since we are primarily interested in the impact of transition scenario uncertainty on the electricity market dynamics, and the price of carbon is the key variable distinguishing transition scenarios, we assume in our extension that the carbon price is random and impacts the production costs of conventional power plants as well as the electricity demand. As is usually the case in IAM scenarios, the carbon price should be interpreted as a proxy for all transition costs rather than a specific emission pricing mechanism. To keep the model relatively simple, we still assume that the capacity factors of renewable power plants are independent, but it would be easy to introduce dependence through an additional random factor.  

\begin{figure}
\centerline{\includegraphics[width=\textwidth]{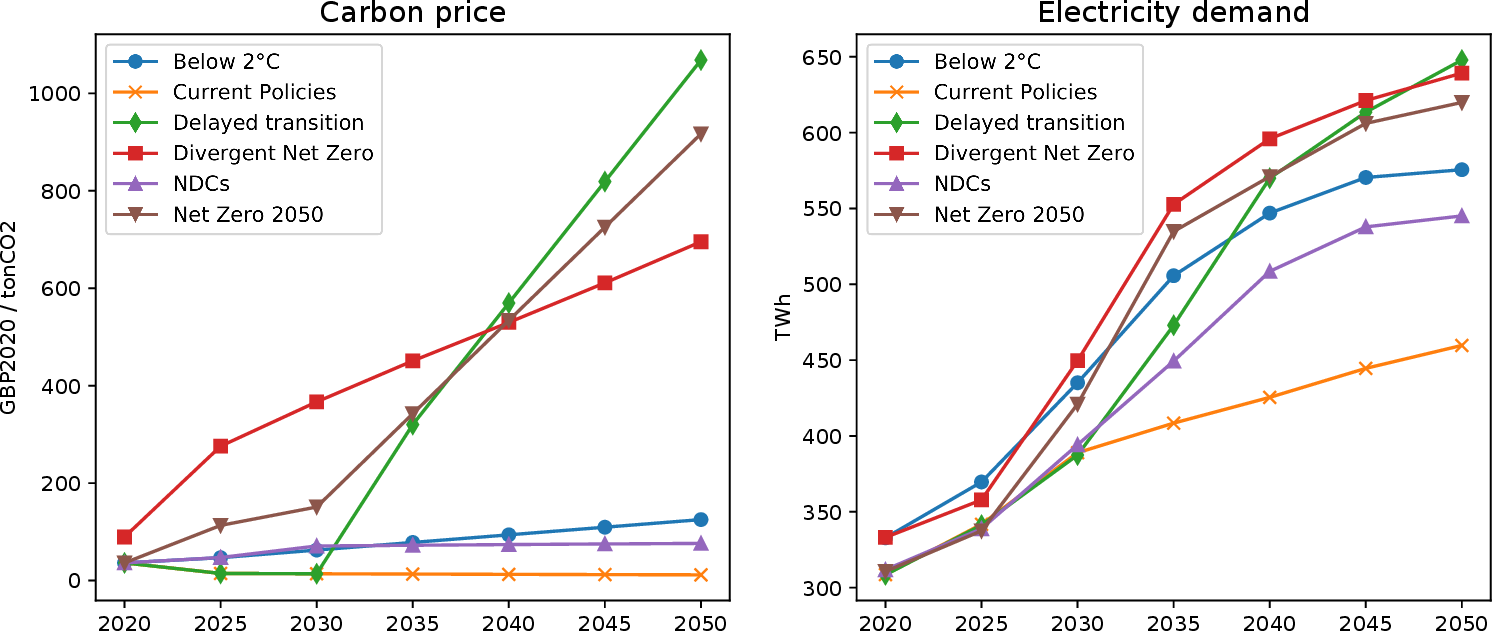}}
\caption{Evolution of carbon price (left) and annual electricity consumption in the 6 scenarios from the NGFS scenario database (REMIND-MAgPIE 3.0-4.4 model downscaled to UK, release 3, available from \url{https://data.ene.iiasa.ac.at/ngfs/}).}
\label{pricedemand}
\end{figure}


\subsection{Model description}\label{sec n player}

We consider a terminal horizon $T_0>0$ (measured in years) and discretize the interval $[0, T_0]$ into $T$ intervals of equal size  $\Delta_t := T_0/T$. 
In the numerical illustrations $\Delta_t$ will be equal to 3 months, allowing to account for annual seasonality. All quantities are understood as averages over period of length $\Delta_t$.  As before, we will denote by $I:=\{0,1, \ldots, T\}$ and $I^*:=\{1, \ldots, T\}$ the sets of time indices. We assume that all random variables considered in this section are defined on the same complete probability space $(\Omega, \mathcal{F}, \mathbb P)$.

\paragraph{Common noise and demand.} The economic agents in our model are subject to common random shocks in the carbon price. We denote by $Z=(Z_t)_{t\in I}$ the stochastic process modeling the carbon price (measured in GBP/tonCO2 in the examples) and suppose that $Z$ takes values in a finite subset $H\subset \mathbb R^+$ and has a deterministic initial value $z_0\in H$.
Moreover we assume that these common shocks depend on the economic scenario, which is not observed by the agents. We consider a finite number of possible scenarios or regimes and we denote by $\mathcal{S}$ the set of all scenarios. Let $S$ be a random variable taking values in $\mathcal{S}$ representing the economic scenario. The law $\Pi_0:=\mathcal{L}(S)\in \mathcal{P}(\mathcal{S})$ is called the prior scenario distribution. Conditional on each scenario, the agents know the distribution of the common random shocks. More precisely, we assume that $(S, Z_t)_{t\in I}$ is a Markov chain with initial condition $\Pi_0 \times \delta_{z_0}$ characterized by the transition kernels $(\pi^Z_t)_{t\in I^*}$ from $\mathcal{S}\times H$ to $H$, i.e.
$\mathbb P(Z_{t}=z'|\mathcal{F}^{S, Z}_{t-1})=\pi_t^Z(S, Z_{t-1}; z')$ a.s.
Finally we assume that the deseasonalized demand process $D$ (averaged over periods of length $\Delta_t$ and measured in GW) is of the form
$$D_t=d(t)+\beta (Z_t-z_0), \quad t\in I,$$
for some deterministic function $d:I\rightarrow \mathbb R^+$ and some parameter $\beta\geq 0$ (in GW$\times$tonCO2/GBP). This means that the demand is positively correlated with the carbon price, which is justified by the fact that as the carbon price increases, carbon-intensive sectors of the industry are forced to electrify and contribute to electricity demand (see Figure \ref{pricedemand}). 

To make our model more realistic, we assume that the demand decomposes as average peak demand $D^p$ (Mon-Fri, 7AM-8PM) and average off-peak demand $D^o$ and we multiply the deseasonalized demand by a deterministic function $(\lambda_t)_{t\in I}\subset\mathbb R^+$, which represents the seasonal cycle correction. More precisely, for all $t\in I$, we have
$$\lambda_t D_t=c_pD_t^p+c_oD_t^o,$$
where $c_p=65/168$ is the proportion of peak hours per week and $c_o=1-c_p$ is the proportion of off-peak hours per week. Moreover we assume that the ratio peak demand/off-peak demand $(D_t^p/D_t^o)_{t\in I}$ is known and constant in time. Denoting by $r_d\geq 1$ this ratio, we can deduce the expressions $D_t^p=\bar c_p \lambda_t D_t$ and $D_t^o=\bar c_o \lambda_t D_t$, where
$\bar c_p:= \frac{r_d}{r_dc_p + c_o}$ and $\bar c_o:= \frac{1}{r_dc_p + c_o}$. $D^p_t$ thus denotes the average demand during peak periods of the time interval $[t\Delta_t,(t+1)\Delta_t]$, and $D^{o}_t$ thus denotes the average demand during off-peak periods of this time interval. 

\paragraph{Conventional producers.} The state of each conventional producer will be described by its baseline marginal cost per MWh produced (in GBP/MWh). In the model of \cite{adt2021}, the baseline marginal cost process follows the CIR dynamics:
\begin{align}
dB^i_t = k(\theta - B^i_t) dt + \delta \sqrt{B^i_t} dW^i_t.\label{adtbaseline}
\end{align}
In the present setting, we approximate the dynamics \eqref{adtbaseline} with a discrete-time Markov chain. Namely, for $\Delta x>0$ and $C_{\text{max}}>C_{\text{min}}\geq 0$, we define the baseline marginal cost processes of the conventional producers $(B^i)_{1\leq i\leq N_C}$ to be independent homogeneous Markov chains on the space $E_1 = \{x_k:=C_{\text{min}} + k \Delta x: k=0, \ldots, n_1\}$ (with $n_1$ such that $C_{\text{min}} + n_1 \Delta x=C_{\text{max}}$) with i.i.d. initial conditions $(\xi^i)_{1\leq i\leq N_C}$ of law $m_0^*$ and with the same transition kernel. The transition probabilities are defined as follows: if $x\in \{x_1, \ldots, x_{n_1-1}\}$, we set
\begin{align*}\pi^B(x, x+\Delta x)&:= \frac{\sigma^2(x)/2 + \Delta x \max (b(x), 0)}{\sigma^2(x)+\Delta x|b(x)|},\\ \pi^B(x, x-\Delta x)&:= \frac{\sigma^2(x)/2 - \Delta x \min (b(x), 0)}{\sigma^2(x)+\Delta x|b(x)|},\end{align*}
otherwise we set $\pi^B(x_0, x_1):=1$ and $\pi^B(x_{n_1}, x_{n_1-1}):=1$ (i.e. the Markov chain reflects on the boundary). Here we take $b(x)=k(\theta - x)$ and $\sigma(x)=\delta \sqrt{x}$ to approximate the continuous-time dynamics \eqref{adtbaseline} (see p. 96 in \cite{kushner2001}).


\vspace{5pt}
When operating at fraction $\alpha$ of its total capacity\footnote{The installed productive capacity is assumed to be the same for all the conventional producers and fixed over time.}, the conventional producer $i\in \{1, \ldots, N_C\}$  has a marginal unit cost (in GBP/MWh) of
$$C_t^i(\alpha):=B^i_t + \tilde\beta Z_t+c(\alpha),\quad t\in I,$$
where $\tilde \beta \geq 0$ (in tonCO2/MWh) is the emission intensity and $c:[0, 1]\rightarrow\mathbb R_+$ is a deterministic $C^1$ function with $c(0) = 0$ and positive derivative on $]0, 1[$. For each step $t\in I$ the conventional producer $i$ observes peak and off-peak electricity price per unit\footnote{This price refers to peak/off-peak electricity futures price with delivery period $[t,t+\Delta_t]$, observed at $t$.}, measured in GBP/MWh, and denoted by $P^p_t$ and $P^{o}_t$, which applies to the entire interval $[t\Delta_t,(t+1)\Delta_t]$. The producer then chooses peak and off-peak capacity utilization rates $\alpha^p_t\in [0,1]$ and $\alpha^{o}_t\in [0,1]$, for the entire period, in order to maximize its profit per unit of capacity (in GBP/MWh): the average profit for peak hours is
\begin{equation}\label{opti capa}
P^p_t\alpha^p_t - \int_0^{\alpha^p_t} C_t^i(y)dy = (P^p_t-B_t^i-\tilde \beta Z_t)\alpha^p_t - \int_0^{\alpha^p_t} c(y)dy,  
\end{equation}
and similarly for off-peak hours.

Introduce the function $\bar \alpha:\mathbb R\rightarrow\mathbb [0, 1]$ defined by
$$
\bar \alpha(y):=
\begin{cases}
0 &\text{ if } y\leq 0,\\
c^{-1}(y) &\text{ if } y\in ]0, c(1)],\\
1 &\text{ if } y> c(1).
\end{cases}
$$
Then $\bar \alpha$ is continuous and it is $C^1$ and increasing on $]0, c(1)[$. Moreover the unique solution of the maximization problem \eqref{opti capa} is given by $\alpha^\star = \bar \alpha(P^p_t-B_t^i-\tilde \beta Z_t)$. The conventional producer $i$ will therefore use the fraction $\bar \alpha(P^p_t-B_t^i-\tilde \beta Z_t)$ of its capacity, making a gain of $G(P^p_t-B_t^i-\tilde \beta Z_t)$ per unit time in peak hours, where
$$G(x):=x\bar\alpha(x)-\int_0^{\bar \alpha (x)}c(y)dy= \int_0^{\bar \alpha (x)}(x-c(y))dy = \int_0^x \bar\alpha(z)dz,$$
the third equality coming from the change of variables $z=c(y)$ and integration by parts.
For the entire period, the average gain per unit time is therefore equal to
$$
c_p G(P^p_t-B_t^i-\tilde \beta Z_t)+c_{o} G(P^{o}_t-B_t^i-\tilde \beta Z_t).
$$

\paragraph{Renewable producers.} We consider a population of renewable producers who generate electricity with zero cost but with a stochastic capacity factor\footnote{As for the conventional producers, the installed productive capacity of the renewable producers is assumed to be the same and fixed over time.}. In the model of \cite{adt2021}, the stochastic capacity factor of renewable producers follows the Jacobi dynamics
\begin{align}
dR^i_t = \bar k(\bar \theta - R^i_t) dt + \bar\delta \sqrt{R^i_t(1-R^i_t)} d\overline W^i_t\label{adtrenewable}
\end{align}
In the present setting, we once again approximate the dynamics \eqref{adtrenewable} with a discrete-time Markov chain. For  $\Delta \bar x>0$ and $1\geq R_{\text{max}}>R_{\text{min}}\geq 0$, we define the stochastic capacity factors of the renewable producers $(R^i)_{1\leq i\leq N_R}$ to be independent homogeneous Markov chains on the space $E_2 = \{\bar x_k:=R_{\text{min}} + k \Delta \bar x: k=0, \ldots, n_2\}$ (with $n_2$ such that $R_{\text{min}} + n_2 \Delta \bar x=R_{\text{max}}$) with i.i.d. initial conditions $(\bar\xi^i)_{1\leq i\leq N_R}$ of law $\bar m_0^*$ and with the same transition kernel. The transition probabilities are defined as follows: if $\bar x\in \{\bar x_1, \ldots, \bar x_{n_2-1}\}$, we set
\begin{align*}\pi^R(\bar x, \bar x+\Delta \bar x)&:= \frac{\bar \sigma^2(\bar x)/2 + \Delta \bar x \max (\bar b(\bar x), 0)}{\bar \sigma^2(\bar x)+\Delta \bar x|\bar b(\bar x)|},\\\quad \pi^R(\bar x, \bar x-\Delta \bar x)&:= \frac{\bar \sigma^2(\bar x)/2 - \Delta \bar x \min (\bar b(\bar x), 0)}{\bar \sigma^2(\bar x)+\Delta \bar x|\bar b(\bar x)|},\end{align*}
otherwise we set $\pi^R(\bar x_0, \bar x_1):=1$ and $\pi^R(\bar x_{n_2}, \bar x_{n_2-1}):=1$. Here we consider $\bar b(x)=\bar k(\bar \theta - x)$ and $\bar \sigma(x)=\bar \delta \sqrt{x(1-x)}$ to approximate the continuous-time dynamics \eqref{adtrenewable}. Recall that the process $R^i_t$ models the capacity factor averaged over periods of length $\Delta_t$; the randomness thus accounts for longer-term climate variability rather than high-frequency wind speed fluctuations. 
For renewable production, we do not make a distinction between peak and off-peak periods: we simply assume that the total output of $i$-th renewable producer per unit of installed capacity during the peak hours of the period $[t\Delta_t, (1+t)\Delta_t]$ is equal to $\Delta_t c_p R^i_t$, and the total output during off-peak hours is given by $\Delta_t c_o R^i_t$. Admittedly, this is not very realistic, as, for example, wind is stronger during off-peak hours. This could be addressed by including correction factors, but, to keep our stylized model tractable, we leave this for further research. 

\paragraph{Information and decision variables.} We assume that the information available to all economic agents is the same and is represented through the complete natural filtration of $(B^i)_{1\leq i\leq N_C}$, $(R^i)_{1\leq i \leq N_R}$ and $Z$, which is denoted by $\mathbb F=(\mathcal{F}_t)_{t\in I}$. We denote by $\mathcal{T}$ the set of $\mathbb F$-stopping times valued in $I$. The conventional producer $i\in \{1, \ldots, N_C\}$ aims to choose a stopping time denoted by $\tau_i\in \mathcal{T}$ to \textit{exit} the market in order to maximize a gain functional to be defined later. In contrast, the (potential) renewable producer $i\in \{1, \ldots, N_R\}$ aims to choose a stopping time denoted by $\bar \tau_i\in \mathcal{T}$ to build the renewable power plant and \textit{enter} the market.

\paragraph{Electricity supply and price formation.} We assume that the installed productive capacity of each conventional producer is $I_C/N_C$, where $I_C>0$ (in GW) is the total installed productive capacity of all conventional producers (who may leave the market) at time $0$. We consider a third class of economic agents who are conventional producers that do not leave the market at any time and produce a baseline electricity supply. These are for example nuclear power plants and more generally any power plants that are unlikely to leave the market during the simulation period. For a given price $p\geq 0$, the baseline supply of the conventional producers is $S^b(p)$ (in GW), where $S^b:\mathbb R_+\rightarrow\mathbb R_+$ is an increasing continuous function with $S^b(0)=0$ and such that there exists some $\tilde p\geq 0$ verifying 
$$S^b(\tilde p)\geq \bar c_p \lambda_{max}(d_{max} + \beta (z_{max}-z_0)),$$
with $d_{max}:=\max_{t\in I}d(t)$, $\lambda_{max}:=\max_{t\in I}\lambda_t$ and $z_{max}:=\max_{z\in H}z$. This is a technical assumption which guarantees that for a sufficiently high price, there is enough baseline supply to satisfy the maximum possible demand level. In our price formation mechanism we will however impose a price cap to be able to model electricity shortages.
At a given price level $p\geq 0$ at step $t\in I$, the total supply of conventional producers per unit time is given by
\begin{align}
S^c_t(p)=\sum_{i=1}^{N_C}\frac{I_C}{N_C}\bar \alpha(p-B^i_t-\tilde \beta Z_t)\mathds{1}_{t<\tau_i} + S^b(p).\label{capconv.eq}
\end{align}

We assume also that there is another class of renewable producers who are already in the market at time $0$. These are simply the renewable plants which have already been constructed at the start of simulation. We suppose that there is a total number of $\tilde N_R$ of these renewable producers and each one has an installed productive capacity of $I^b_R/\tilde N_R$, where $I^b_R>0$ (in GW) is the total installed productive capacity of these renewable producers. Furthermore we assume that their stochastic capacity factors are given by i.i.d. processes $(\tilde R^i)_{1\leq i\leq \tilde N_R}$ with the same distribution as $R^1$ and which are independent from the other processes. Finally, the installed productive capacity of each potential renewable producer (i.e. not yet in the market) aiming to enter the market will be $I_R/N_R$, where $I_R>0$ (in GW) is the maximum renewable capacity which can potentially be added until time $T$. The total supply of renewable producers per unit time at step $t\in I$ is
\begin{align}
S^r_t=\sum_{i=1}^{\tilde N_R}\frac{I_R^b}{\tilde N_R}\tilde R_t^i + \sum_{i=1}^{N_R}\frac{I_R}{N_R}R^i_t\mathds{1}_{t\geq \bar \tau_i}.\label{capren.eq} 
\end{align}

The market price (in GBP/MWh) is determined by the merit order mechanism, applied separately to peak and off-peak periods.\footnote{Recall that we are not dealing with day-ahead prices but with futures prices with delivery period $\Delta_t$. Our merit order mechanism consists simply in matching of demand and supply in the futures market. Electricity futures contracts are traded separately for peak and off-peak periods.} The renewable producers bid their entire (stochastic) capacity at zero price. If the renewable supply is not sufficient to match the demand, the conventional suppliers fill the residual demand. We assume that the price is capped to some value $p_{max}>0$. Consider first the peak hours of the period $[t\Delta_t, (1+t)\Delta_t]$. The total demand for these hours is given by $\Delta_t c_p D^p_t$, the renewable supply is $\Delta_t c_p S^r_t$, and the conventional supply at price level $p$ is $\Delta_t c_pS^c_t(p)$.

The peak price for this interval is thus given by the random variable
$$P_t^p:=\inf\{p\geq 0: (D_t^p-S^r_t)^+\leq S_t^c(p)\}\wedge p_{max}.$$
Observe that the set $\{p\geq 0: (D_t^p-S^r_t)^+\leq S_t^c(p)\}$ is nonempty since by assumption $\tilde p$ belongs to it. The off-peak price $P^o$ is defined in the same way, but replacing the peak demand $D^p$ by the off-peak demand $D^o$. \vspace{5pt}

We remark that the producers impact the price  through the empirical measures associated to the state processes and stopping times of the two populations of agents:
$$\frac{1}{N_C}\sum_{i=1}^{N_C}\delta_{B^i_t}(dx)\mathds{1}_{t<\tau_i}, \quad \frac{1}{N_R}\sum_{i=1}^{N_R}\delta_{R^i_t}(dx)\mathds{1}_{t\geq \bar \tau_i}.$$
This structure will allow us to simplify the expression of the price in the mean-field limit (i.e. when the number of agents in each population goes to infinity). Because of the common random shocks that do not average out in the limit, the price will be an adapted functional of the stochastic carbon price $Z$.

\paragraph{Expected gain of conventional and renewable producers.} The conventional producer $i\in \{1, \ldots, N_C\}$ aims to maximize its expected discounted profit (in GBP per MW of installed capacity) over the time period $[0, T_0]$ by choosing when to exit the market, i.e., solves the following problem:
\begin{multline*}
\sup_{\tau_i\in \mathcal{T}}\mathbb E\Big[  \sum_{t=0}^{\tau_i-1}e^{-\rho t\Delta_t} \left(c_pG(P_t^p-B_t^i-\tilde \beta Z_t) + c_oG(P_t^o-B_t^i-\tilde \beta Z_t)- \kappa_C\right)\Delta_t\\+K_Ce^{-(\gamma_C+\rho)\tau_i\Delta_t}\Big], 
\end{multline*}
where $\kappa_C\geq 0$ is the fixed operating cost (in GBP/MWyr), $e^{-\gamma_C \tau_i\Delta_t}K_C$ is the value recovered per MW of installed capacity (in GBP/MW) if the plant is sold at the year $\tau_i\Delta_t$ with depreciation rate $\gamma_C$ and $\rho$ is the discount rate.

On the other hand, the renewable producer $i\in \{1, \ldots, N_R\}$ bids its full (intermittent) capacity after entering the market, and wants to maximize its expected profit (in GBP per MW of installed capacity) over the time period $[0, T_0]$ by choosing when to enter the market, i.e.
$$\sup_{\bar\tau_i\in \mathcal{T}}\mathbb E\left[\sum_{t = \bar\tau_i}^{T-1} e^{-\rho t\Delta_t}([c_pP_t^p+c_oP_t^o] R^i_t-\kappa_R)\Delta_t + K_Re^{-\rho T_0-\gamma_R(T_0-\bar \tau_i\Delta_t)} - K_R e^{-\rho \bar \tau_i\Delta_t}\right],$$
where $\kappa_R\geq 0$ is the fixed operating cost (in GBP/MWyr), $K_R$ is the fixed cost per kW of building the plant (in GBP/MW), 
$e^{-\gamma_R(T_0-\bar \tau_i\Delta_t)}K_R$ is the value per kW of the plant at time $T_0$ (in GBP/MW) with $\gamma_R$ being the depreciation rate. 

\paragraph{Nash equilibria.} Since the expected gain functionals depend on the electricity price, which is determined by the decisions of both classes of agents, all optimization problems are coupled and it is natural to look for Nash equilibrium strategies. If the conventional producers choose the stopping times $\bm{\tau}=(\tau_1, \ldots, \tau_{N_C})\in \mathcal{T}^{N_C}$ and the renewable producers choose the stopping times $\bar{\bm{\tau}}=(\bar \tau_1, \ldots, \bar \tau_{N_R})\in \mathcal{T}^{N_R}$, then we denote by $V_i(\bm{\tau}, \bar{\bm{\tau}})$ the corresponding expected profit per unit of capacity of the conventional producer $i\in \{1, \ldots, N_C\}$ and by $\bar V_j(\bm{\tau}, \bar{\bm{\tau}})$ the corresponding expected profit per unit of capacity of the renewable producer $j\in \{1, \ldots, N_R\}$. We will say that $(\bm{\tau}^\star, \bar{\bm{\tau}}^\star)$ is a Nash equilibrium if for all $i\in \{1, \ldots, N_C\}$ and $\tau_i\in \mathcal{T}$,
$$V_i((\tau_i, \bm{\tau}^{\star,-i}), \bar{\bm{\tau}}^\star)\leq V_i(\bm{\tau}^\star, \bar{\bm{\tau}}^\star),$$
and for all $j\in \{1, \ldots, N_R\}$ and $\bar \tau_j\in \mathcal{T}$,
$$\bar V_j(\bm{\tau}^\star, (\bar \tau_j, \bar{\bm{\tau}}^{\star, -j}))\leq \bar V_j(\bm{\tau}^\star, \bar{\bm{\tau}}^\star).$$
Here we used the notation $(y, x^{-i}):=(x_1, \ldots, x_{i-1}, y, x_{i+1}, \ldots, x_d)$ for some $y\in \mathbb R$ and $x\in \mathbb R^d$.
In other words, a Nash equilibrium is a tuple of strategies (stopping times in our case) where no agent has an incentive to deviate from its individual strategy. 


Our main goal is to understand the evolution of the market in equilibrium (i.e. when players choose the Nash equilibrium strategies). In particular, we want to understand the renewable penetration under common random shocks on the carbon price and under scenario uncertainty. To study our model we make use of MFGs with common noise and partial information, which are introduced in the next section.

\subsection{MFGs formulation}
\label{model_mfg}
The game theoretic framework for day-ahead electricity markets described in Section \ref{sec n player} is not tractable for a large number of conventional and renewable producers. However, since the interactions of the agents via the price are of the mean-field type and there is symmetry of the agents dynamics inside each class, we can expect that the problem is well approximated by an MFG with common noise and two populations. Here the common noise is the process $Z$ interpreted as  the carbon price. Note that $Z$ is not a Markov process in general due to its dependence on the economic scenario, which, to the best of our knowledge, contrasts with the literature on MFGs with common noise. From now on we consider the asymptotic formulation. We place ourselves in a complete probability space $(\Omega, \mathcal{F}, \mathbb P)$ big enough to carry all the random variables described below.

\paragraph{Probabilistic set-up.} In the MFG setting we consider a representative agent for each class. The representative conventional producer is characterized by its baseline marginal cost per MWh $B$ which is a homogeneous Markov chain on $E_1$ with initial distribution $m_0^*$ and transition kernel $\pi^B$. On the other hand, the stochastic capacity factor of the representative renewable producer is given by a homogeneous Markov chain $R$ on $E_2$ with initial distribution $\bar m_0^*$ and transition kernel $\pi^R$. We consider the common noise process $Z$, the economic scenario $S$ and the normalized demand process $D$ as described in Section \ref{sec n player}. We assume that $B$, $R$ and $(S, Z)$ are independent. We endow the probability space with the completed natural filtration of $B$, $R$ and $Z$, denoted by $\mathbb F=(\mathcal{F}_t)_{t\in I}$. We denote by $\mathcal{T}$ the set of $\mathbb F$-stopping times valued in $I$. Construct $U$, $W$, $(\Omega_t)_{t\in I}$ and $(\Psi_t)_{t\in I}$ as in Section \ref{abstract}.

\paragraph{Electricity supply and price formation.} Define the supply mappings $S^c:H\times\mathcal{P}^{sub}(E_1)\times \mathbb R_+\rightarrow\mathbb R$ and $S^r: I\times \mathcal{P}^{sub}(E_2)\rightarrow\mathbb R$ by
$$S^c(z, m, p):=I_C \int_{E_1}\bar\alpha(p - x-\tilde \beta z)m(dx) + S^b(p),$$
$$S^r(t, \bar m):= (I_R^b + I_R)\int_{E_2}x\;\eta^R_t(dx)- I_R\int_{E_2}x\;\bar m(dx),$$
where $(t, z, m, \bar m, p)\in I\times H\times  \mathcal{P}^{sub}(E_1)\times \mathcal{P}^{sub}(E_2)\times \mathbb R_+$ and $\eta^R_t=\mathcal{L}(R_t)$. Here $m$ corresponds to the distribution of the baseline cost of the representative conventional producer who is still in the game and $\bar m$ corresponds to the distribution of the stochastic capacity factor of the representative potential renewable project who has not yet made the decision to build the plant. Recall that $I_C$ is the total initial capacity of conventional plants which may leave the market, $I_R$ is the total potential capacity of renewable plants which may enter the market but are not yet constructed at time $t=0$, and $I^b_R$ is the total capacity of renewable plants which are already in the market at time $t=0$. These mappings should be compared with equations \eqref{capconv.eq} and \eqref{capren.eq} in the setting of finite number of agents.

Define the price function $P:I\times H\times \mathbb R_+^2\times  \mathcal{P}^{sub}(E_1)\times \mathcal{P}^{sub}(E_2)\rightarrow [0, p_{max}]^2$ by 
$$P(t, z, d, m, \bar m)=(\bar P(t, z, d_1, m, \bar m), \bar P(t, z, d_2, m, \bar m)),$$ 
where $\bar P:I\times H\times \mathbb R_+\times  \mathcal{P}^{sub}(E_1)\times \mathcal{P}^{sub}(E_2)\rightarrow [0, p_{max}]$ is defined by
$$\bar P(t, z, d, m, \bar m)=\inf\{p\geq 0: (d-S^r(t, \bar m))^+\leq S^c(z, m, p)\}\wedge p_{max}.$$
The first component of $P$ will correspond to the peak price and the second component to the off-peak price.

\paragraph{Reward mappings.} Define the reward functions for the conventional producer $f:I\setminus\{T\}\times E_1\times H\times \mathbb R_+^2\rightarrow \mathbb R$ and $g:I\rightarrow \mathbb R$ by 
\begin{align*}
f(t, x, z, p)&:=e^{-\rho t\Delta_t}(c_p G(p_1-x-\tilde \beta z) + c_o G(p_2-x-\tilde \beta z) -\kappa_C)\Delta_t,\\
g(t)&:=K_C e^{-(\gamma_C+\rho) t \Delta_t}.
\end{align*}
Define also the reward functions for the renewable producer $\bar f:I\setminus\{T\}\times E_2\times \mathbb R_+^2\rightarrow \mathbb R$ and $\bar g:I\rightarrow \mathbb R$ by 
\begin{align*}
\bar f(t, x, p)&:= - e^{-\rho t\Delta_t}([c_p p_1 + c_o p_2] x- \kappa_R)\Delta_t, \\
\bar g(t)&:=K_Re^{-\rho T_0-\gamma_R(T_0- t\Delta_t)} - K_R e^{-\rho t\Delta_t}.
\end{align*}
Observe that we defined $\bar f$ with a minus sign. This will be justified in the definition of the strong formulation for the MFG problem.

\paragraph{Strong MFG equilibrium with strict stopping time.} We give here the definition of a strong MFG equilibrium with strict stopping time for this particular problem with respect to the initial probability space. 

\begin{definition}[\textit{Strong MFG equilibrium with strict stopping time}]\leavevmode
\begin{enumerate}[(1)]
\item Fix for each $t\in I\setminus\{T\}$,  $m_t:W\rightarrow \mathcal{P}^{sub}(E_1)$ and $\bar m_t:W\rightarrow \mathcal{P}^{sub}(E_2)$, and find the solution to the optimal stopping problems
\begin{equation}\label{opti strong 1}
\sup_{\tau \in \mathcal{T}}\; \mathbb{E}\left[\sum_{t=0}^{\tau-1} f\left(t, B_t, Z_t, P(t, Z_t, D_t^p, D_t^o, m_t(U_t), \bar m_t(U_t))\right) + g\left(\tau\right)\right],
\end{equation}
\begin{equation}\label{opti strong 2}
\sup_{\bar\tau \in \mathcal{T}}\; \mathbb{E}\left[\sum_{t=\bar\tau}^{T-1} (-\bar f)\left(t, R_t, P(t, Z_t, D_t^p, D_t^o, m_t(U_t), \bar m_t(U_t))\right) + \bar g\left(\bar \tau\right)\right].
\end{equation}
Remark that problem \eqref{opti strong 2} is equivalent to the following problem
\begin{equation}\label{opti strong 2 bis}
\sup_{\bar\tau \in \mathcal{T}}\; \mathbb{E}\left[\sum_{t=0}^{\bar\tau-1} \bar f\left(t, R_t, P(t, Z_t, D_t^p, D_t^o, m_t(U_t), \bar m_t(U_t))\right) + \bar g\left(\bar \tau\right)\right],
\end{equation}
which is the one that we will use later.
\item Denoting by $\tau^{m, \bar m}$ and $\bar \tau^{m, \bar m}$ some optimal stopping times associated to $(m, \bar m)$ (solutions of the problems \eqref{opti strong 1} and \eqref{opti strong 2 bis} respectively), find $(m, \bar m)$ such that
$$m_t(B) = \mathbb P[B_t\in B, t<\tau^{m, \bar m}|U_t = u], \quad B\in \mathcal{B}(E_1),\quad t\in I\setminus\{T\},\quad u\in \Omega_t,$$
$$\bar m_t(B) = \mathbb P[R_t\in B, t<\bar \tau^{m, \bar m}|U_t = u], \quad B\in \mathcal{B}(E_2),\quad t\in I\setminus\{T\},\quad u\in \Omega_t.$$
\end{enumerate}
\end{definition}

\noindent This problem is difficult to solve for general coefficients, so we will work with the linear programming formulation presented in the previous section. This formulation allows to prove existence of equilibria under general assumptions and facilitates their numerical computation.

\paragraph{Linear programming formulation.} 

As explained in the previous section, the linear programming formulation consists in replacing the expectations concerning the processes and stopping times by integrals of occupation measures. Using a similar definition to Definition \ref{Set of constraints}, we denote by $\mathcal{R}$ (resp. $\bar{\mathcal{R}}$) the set of pairs $(\mu, m)$ (resp. $(\bar\mu, \bar m)$) satisfying the LP constraint associated to $(B, U)$ (resp. $(R, U)$).

\noindent Now define the peak and off-peak demand functions: for $t\in I$ and $u\in W$,
$$d_t^p(u):=\bar c_p\lambda_t(d(t) + \beta (\Psi_t(u)-z_0)), \quad d_t^o(u):=\bar c_o\lambda_t(d(t) + \beta (\Psi_t(u)-z_0)),$$
where $\Psi_t$ is defined by \eqref{def psi}, and maps $U_t$ to $Z_t$.

\begin{definition}
For $(\mu', m',\bar \mu', \bar m')\in \mathcal{R}\times \bar{\mathcal{R}}$, let $\Gamma[\mu', m',\bar \mu', \bar m']: \mathcal{R}\rightarrow \mathbb R$ be the reward functional of the conventional producer associated to $(\mu', m',\bar \mu', \bar m')$ and defined by
\begin{multline*}
\Gamma[\mu', m',\bar \mu', \bar m'] (\mu, m) = \sum_{t=0}^T\sum_{u\in \Omega_t}p_t(u)\int_{E_1}g(t)\mu_t(u)(dx) + \\\sum_{t=0}^{T-1}\sum_{u\in \Omega_t}p_t(u)\int_{E_1}f(t, x, \Psi_t(u), P(t, \Psi_t(u), d_t^p(u), d_t^o(u), m_t'(u), \bar m_t'(u)))m_t(u)(dx) .    
\end{multline*}
Analogously, we define the reward functional of the renewable producer associated to $(\mu', m',\bar \mu', \bar m')$, $\bar\Gamma[\mu', m',\bar \mu', \bar m']: \bar{\mathcal{R}}\rightarrow \mathbb R$, by
\begin{multline*}
\bar\Gamma[\mu', m',\bar \mu', \bar m'] (\bar\mu, \bar m) = \sum_{t=0}^T\sum_{u\in \Omega_t}p_t(u)\int_{E_2}\bar g(t)\bar \mu_t(u)(dx) \\+ \sum_{t=0}^{T-1}\sum_{u\in \Omega_t}p_t(u)\int_{E_2}\bar f(t, x, P(t, \Psi_t(u), d_t^p(u), d_t^o(u), m_t'(u), \bar m_t'(u)))\bar m_t(u)(dx).    
\end{multline*}
\end{definition}

\noindent We aim to apply Theorem \ref{nash K} with two populations in order to show the existence of an LP MFG Nash equilibrium in our framework. Let us verify Assumption \ref{assump existence K pop}. In our setting we do not have absorption of the state processes so that $D_1=E_1$ and $D_2=E_2$. The transition kernels are continuous since we work with finite sets endowed with the discrete topology. Let us show the continuity of the reward functions. The function $f$ is continuous using that $G$ is continuous and the functions $\bar f$, $g$ and $\bar g$ are obviously continuous. Now the function $(t, z, d, m, \bar m) \mapsto \bar P(t, z, d, m, \bar m)$ is continuous by Lemma \ref{price cont}. Assumption \ref{assump existence K pop} is verified and by Theorem \ref{nash K} there exists an LP MFG Nash equilibrium. We present now a uniqueness result for the price in an equilibrium situation. The proof follows the same ideas as in \cite{adt2021} and is therefore omitted to save space.

\begin{proposition}
Assume that the function $S^b$ is $C^1$ on $\mathbb R_+$ and its derivative is bounded below by some constant $c_b>0$. If $(\mu^1, m^1, \bar\mu^1, \bar m^1), (\mu^2, m^2, \bar\mu^2, \bar m^2) \in \mathcal{R}\times \bar{\mathcal{R}}$ are two LP MFG Nash equilibria, then for all $t\in I\setminus\{T\}$ and $u\in \Omega_t$,
\begin{align*}
P(t, \Psi_t(u), d_t^p(u), d_t^o(u), m_t^1(u), \bar m_t^1(u)) = P(t, \Psi_t(u), d_t^p(u), d_t^o(u), m_t^2(u), \bar m_t^2(u)).
\end{align*}
\end{proposition}

\subsection{Numerical computation via the Linear Programming Fictitious Play algorithm}
\label{numerics}
The Linear Programming Fictitious Play algorithm is presented in detail in \cite{dlt2022} for the case of continuous time MFGs without common noise and with one class of agents. We present here an algorithm (Algorithm \ref{algo lpfp}) for discrete time MFGs in the common noise setting and with two populations. 

\begin{algorithm}
\caption{LPFP algorithm}\label{algo lpfp}
\SetAlgoLined
\KwData{A number of steps $N_{iter}$ for the equilibrium approximation; a pair $(\mu^{(0), g}, m^{(0), g})\in \mathcal{R}$; a pair $(\bar\mu^{(0), g}, \bar m^{(0), g})\in \bar{\mathcal{R}}$}
\KwResult{Approximate LP MFG Nash equilibrium}
\For{$\ell=0, 1, \ldots, N_{iter}-1$}{
Compute a linear programming best response $(\mu^{(\ell+1)}, m^{(\ell+1)})$ to $(\mu^{(\ell), g}, m^{(\ell), g}, \bar\mu^{(\ell), g}, \bar m^{(\ell), g})$ by solving the linear programming problem
$$\argmax_{(\mu, m)\in \mathcal{R}}\Gamma[\mu^{(\ell), g}, m^{(\ell), g}, \bar\mu^{(\ell), g}, \bar m^{(\ell), g}](\mu, m).$$
\\
Compute a linear programming best response $(\bar\mu^{(\ell+1)}, \bar m^{(\ell+1)})$ to $(\mu^{(\ell), g}, m^{(\ell), g}, \bar\mu^{(\ell), g}, \bar m^{(\ell), g})$ by solving the linear programming problem
$$\argmax_{(\bar\mu, \bar m)\in \bar{\mathcal{R}}}\bar\Gamma[\mu^{(\ell), g}, m^{(\ell), g}, \bar\mu^{(\ell), g}, \bar m^{(\ell), g}](\bar\mu, \bar m).$$
\\
Set $(\mu^{(\ell+1), g}, m^{(\ell+1), g}):=\frac{\ell}{\ell+1}(\mu^{(\ell), g}, m^{(\ell), g}) + \frac{1}{\ell+1}(\mu^{(\ell+1)},  m^{(\ell+1)})$. 
\\
Set $(\bar\mu^{(\ell+1), g}, \bar m^{(\ell+1), g}):=\frac{\ell}{\ell+1}(\bar \mu^{(\ell), g}, \bar m^{(\ell), g}) + \frac{1}{\ell+1}(\bar \mu^{(\ell+1)},  \bar m^{(\ell+1)})$. 
}
\end{algorithm}

The computation of the best responses in the algorithm is done by solving the linear programming problems. Observe that by linearity, the constraint defining the set $\mathcal{R}$ (resp. $\bar{\mathcal{R}}$) is satisfied for all $\varphi\in C(I\times E_1\times W)$ (resp. $C(I\times E_2\times W)$) if and only if it is satisfied for the subset of indicator functions, which gives us a finite number of constraints. Henceforth, the algorithm consists in iteratively solving finite dimensional linear programming problems, for which there exist fast solvers\footnote{see \url{http://plato.asu.edu/ftp/lpsimp.html} for a comparison of different LP solvers.}. We solve the LP problems by using the Gurobi\footnote{\url{https://www.gurobi.com/}} solver in Python (Version 9.5.1). 

\section{Illustration}
\label{illustration}
\subsection{Parameters}

We specify here all the parameters used in our model. We refer to \cite{adt2021} for the calibration of the parameters non-related to the carbon price.

\paragraph{Time:} We set the time horizon to $T_0=18$ years and the time steps to $3$ months, i.e. $T= 72$ and $\Delta_t=0.25$. We set the discount rate to $\rho=0.086$.

{\paragraph{Common noise:} To cover the range of carbon prices across the scenarios, we fix the possible values of the carbon price to be $H=\{50, 75, 100, 125, 150, 175, 200\}$ and denote its elements by $z_0<z_1<\ldots<z_6$. We assume that the carbon price $Z$ starts at $z_0$ in every scenario, and then, in each scenario $s\in \mathcal{S}$ at each adjustment date $t\in J$ can stay in the same state with probability $p_t(s)\in [0, 1]$ or jump to the next state with probability $1-p_t(s)$. The adjustment dates of the carbon price are $J = \{10,20,30,40,50,60\}$; with $\Delta_t = 0.25$ this means that the carbon price adjustment happens every $2.5$ years. More precisely, the dynamics of $Z$ are as follows: for $t\in J$ and $i=0,1,\dots5 $, we set $\pi_t^Z(s, z_i;z_{i})=p_t(s)$ and $\pi_t^Z(s, z_i;z_{i+1})=1-p_t(s)$, otherwise we set $\pi_t^Z(s, z_i;z_i)=1$. 
We recall that the carbon price in our model is a proxy for all costs carbon emitting companies face in the context of energy transition, rather than a specific emission trading scheme price, thus it is natural to model it as an increasing process.

In order to quantify the impact of scenario uncertainty, we will consider two different settings. 
\begin{itemize}
\item \textbf{Setting 1:} In the first setting, there are two possible scenarios $\mathcal{S}=\{0, 1\}$ with uniform prior distribution. In the first scenario $S=0$, the probabilities $(p_t(0))_{t\in J}$ are all equal to $0.9$, meaning that with high probability the carbon price will stay low. In the second scenario $S=1$, the probabilities $(p_t(1))_{t\in J}$ are all equal to $0.1$, meaning that with high probability the carbon price will increase at the adjustment dates. 
\item \textbf{Setting 2:} In the second setting, there is only one scenario $\mathcal{S}=\{0\}$, meaning that there is no uncertainty on the scenario, and the probabilities $(p_t(0))_{t\in J}$ are all equal to $0.5$, implying that each carbon price trajectory has equal probability.
\end{itemize}}

\paragraph{Conventional dynamics:} Following \cite{adt2021}, we take $C_{\text{min}}=0$, $C_{\text{max}}=70$, $\theta = 33.4-\tilde\beta z_0$, $k = 0.5$, $\delta=\text{std}_C\sqrt{\frac{2k}{\theta}}$ where $\text{std}_C=11$. In the spirit of \cite{kushner2001} p. 96, we choose $\Delta x$ so that
$$\frac{\Delta x^2}{\sigma^2(x) + \Delta x|b(x)|}\approx \Delta_t.$$
We evaluate $\sigma$ and $b$ at $x=\theta$, which is the expectation of the stationary distribution of the CIR process. Since $b(\theta)=0$, we obtain $\Delta x\approx \sigma(\theta)\sqrt{\Delta_t} = \delta \sqrt{\theta\Delta_t}$. We define $n_1=\lfloor (C_{\text{max}}-C_{\text{min}})(\delta \sqrt{\theta\Delta_t})^{-1}\rfloor$ and then $\Delta x=(C_{\text{max}}-C_{\text{min}})/n_1$. We finally define $m_0^*$ as the discretization (by evaluating the density on the grid and then normalizing) of a gamma distribution with shape parameter $2\theta k/\delta^2$ and scale parameter $\delta^2/2k$, which is the stationary distribution of the CIR process.

\paragraph{Renewable dynamics:} Once again, following \cite{adt2021}, we take $R_{\text{min}}=0.3$, $R_{\text{max}}=0.6$, $\bar\theta = 0.43$, $\bar k = 0.5$, $\bar\delta=\text{std}_R\sqrt{\frac{2\bar k}{\bar\theta(1-\bar\theta)-\text{std}_R^2}}$ where $\text{std}_R=0.044$. We define $n_2=\lfloor (R_{\text{max}}-R_{\text{min}})(\bar\delta \sqrt{\bar\theta(1-\bar\theta)\Delta_t})^{-1}\rfloor$ and set $\Delta \bar x=(R_{\text{max}}-R_{\text{min}})/n_2$. We define $\bar m_0^*$ as the discretization of a beta distribution with shape parameters $2\bar k\bar \theta/\bar\delta^2$ and scale parameter $2\bar k(1-\bar \theta)/\bar \delta^2$, which is the stationary distribution of the Jacobi process.

{\paragraph{Demand:} As in the paper \cite{adt2021}, we use the peak/off-peak ratio $r_d=1.29$ and we take the seasonal cycle $\lambda_t$ to be equal to 1.10 in the first quarter, 0.93 in the second quarter, 0.91 in the third quarter and 1.06 in the last quarter. We fix the carbon price sensitivity of the demand $\beta$ in an ad hoc manner to $\beta = 0.015$. Finally, we take the baseline demand level $d$ from the British government demand projections\footnote{\url{https://www.gov.uk/government/publications/updated-energy-and-emissions-projections-2019}, Annex F.}.}

\paragraph{Capacity levels:} As in \cite{adt2021}, we use the conventional baseline supply function $S^b(p)=\frac{12.1}{p_{max}}p$ where $p_{max}=150$, the conventional capacity $I_C=35.9$, installed renewable capacity $I_R^b=35.6$ and potential renewable capacity $I_R=47$.

\paragraph{Conventional rewards:} We set the emission intensity  to $\tilde\beta = 0.429\footnote{see \url{https://www.rte-france.com/en/eco2mix/co2-emissions}}$. As in \cite{adt2021}, we use the fixed operating cost $\kappa_C=30$ (in GBP per kW per year), scrap value of the plant $K_C=0$ and an ad hoc supply function $\bar\alpha:\mathbb R\rightarrow [0, 1]$ defined by
$$\bar\alpha(y)=\mathds{1}_{y> c_{max}}+\frac{1}{2}\left(1+\sin\left(-\frac{\pi}{2} + \frac{\pi}{c_{max}}y\right)\right)\mathds{1}_{0\leq y\leq c_{max}},$$
where $c_{max}=0.5$ corresponds to $c(1)$. This choice of $\bar\alpha$ leads to 
\begin{align*}G(x) =\int_0^x\bar\alpha(y)dy & =\left(x-\frac{c_{max}}{2}\right)\mathds{1}_{x>c_{max}}\\
&+\frac{1}{2}\left(x-\frac{c_{max}}{\pi}\cos\left(-\frac{\pi}{2}+\frac{\pi}{c_{max}}x\right)\right)\mathds{1}_{0\leq x\leq c_{max}}.
\end{align*}

\paragraph{Renewable rewards:} As in \cite{adt2021}, we use the fixed operating cost $\kappa_R=17.21$ GBP per kW per year, investment cost $K_R=1377$ GBP per kW of capacity and depreciation rate $\gamma_R=\ln(2)/10$.

\subsection{Results}

In our setting, there is a total number of $2^6 = 64$ possible trajectories of the carbon price $Z$. We select the lowest carbon price trajectory $\omega_{min}=(50, \ldots, 50)$ and the highest carbon price trajectory $\omega_{max}=(50, 75, 100, 125, 150, 175, 200)$ to plot the quantities of interest conditionally on $Z$ taking these paths. For comparison, we also solve the MFG model without common noise where the carbon price follows the (deterministic) trajectories $\omega_{min}$ and $\omega_{max}$ with probability 1. 
In the latter case, the agents know the exact trajectory of the carbon price which they use for choosing their stopping times, whereas in the stochastic case the realized carbon price trajectory is the same, but, when making their decisions, the agents take into account the risk of a carbon price rise at each date. The difference between the stochastic and deterministic case therefore quantifies the impact of climate policy uncertainty on the agents' behavior.  

In Figure \ref{error} we represent the exploitabilities of the conventional and the renewable producers, that is, the gain of the agents from switching from the current iteration to the best response, in log-log scale (base 10), as function of the number of iterations.  The exploitability is the standard measure of convergence for this type of algorithms (see \cite{dlt2022} for a precise definition in the LP approach), which shows how far the current iteration is from the optimum. For a convergent algorithm, the exploitability should tend to zero as function of the number of iterations, and from the graphs we see, at least empirically, that this happens at the rate of $O(N^{-1})$.

\begin{figure}[h]
    \centering
    \subfloat{%
        \includegraphics[width=0.5\textwidth]{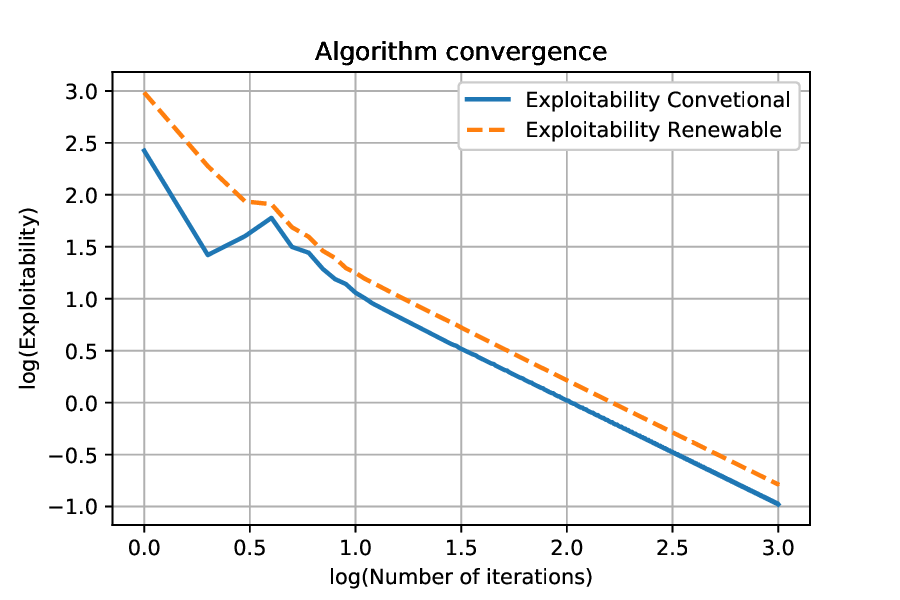}%
        \label{error 1}%
        }%
    \subfloat{%
        \includegraphics[width=0.5\textwidth]{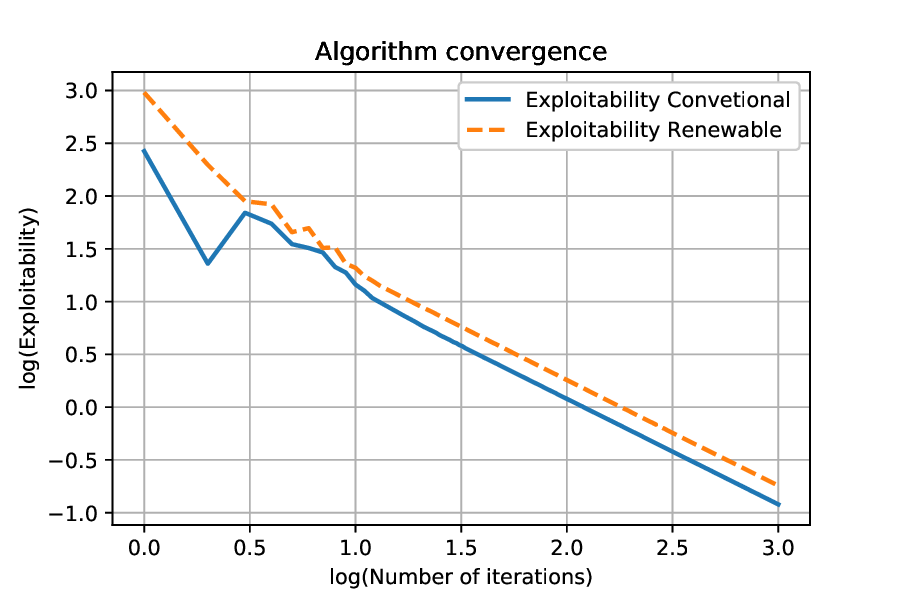}%
        \label{error 2}%
        }%
    \caption{Convergence of the algorithm (left: Setting 1, right: Setting 2): log-log plot (base 10) of the exploitability.}
    \label{error}
\end{figure}

{Figure \ref{capacity} shows the total installed capacity of conventional and renewable producers as function of time for our sample trajectories $\omega_{min}$ and $\omega_{max}$ of the carbon price, for the MFG without common noise, as well as the expectation over all possible trajectories of the carbon price. As expected, higher carbon prices lead to faster growth of renewable penetration and faster departures of conventional producers. In the presence of ``upward" uncertainty (when the carbon price is stochastic and follows the lowest carbon price trajectory $\omega_{min}$), the renewable power plants are built earlier, and the conventional capacity decreases faster than when the price follows the deterministic trajectory $\omega_{min}$, even if the realized carbon price trajectory is the same in both cases. When there is a credible possibility of a carbon price rise, even if this possibility is not realized, conventional producers exit the market to avoid the risk of paying a high carbon price, and renewable producers enter the market to benefit from potential opportunities created by the departure of the conventional producers. Similarly, in the presence of ``downward" uncertainty (when the carbon price is stochastic and follows the highest carbon price trajectory $\omega_{max}$), the conventional producers leave later, and the renewable plants are built later than with a deterministic trajectory. 

Comparing top and bottom graphs of Figure \ref{capacity} illustrates the impact of scenario uncertainty as compared to a random carbon price trajectory within a single scenario. In the top graphs, there are two scenarios, but after the first price adjustment date the uncertainty resolves almost completely (if there is an upward price move, we are very likely to be in the high-price scenario); as a result, after this date the trajectories of conventional and renewable installed capacity are close to the deterministic ones. On the other hand, in the bottom graphs, even after the first adjustment date, the price is still random and the capacities are further away from the ones corresponding to deterministic price trajectories. 


Figure \ref{proportions}   shows the peak and off-peak prices for our sample trajectories $\omega_{min}$ and $\omega_{max}$ of the carbon price, for the MFG without common noise, as well as the expectation over all possible trajectories of the carbon price. As expected, higher carbon prices lead to higher peak electricity prices. Interestingly, the impact on off-peak prices is less clear: higher carbon prices increase renewable penetration, which leads to lower electricity prices in summer, when demand is low, but they also increase the costs of conventional generation, which may push winter prices up, even in off-peak periods. 

Carbon price uncertainty, especially that of the ``upward" kind, although it does not modify the production costs directly, pushes some conventional producers, anticipating potential carbon price hikes, out of the market, which leads to higher peak prices compared to a deterministic trajectory. This effect is stronger in Setting 2, where price uncertainty remains for the entire period than in Setting 1, where uncertainty resolves after the first price adjustment date. On the other hand, off-peak prices in summer are somewhat lower in the presence of carbon price uncertainty than for a deterministic carbon price trajectory. This is due to the fact that in these periods generation is predominantly renewable, and under carbon price uncertainty renewable penetration is higher than for the deterministic price path. }

\begin{figure}[h]
    \centering
    \subfloat{%
        \includegraphics[width=0.565\textwidth]{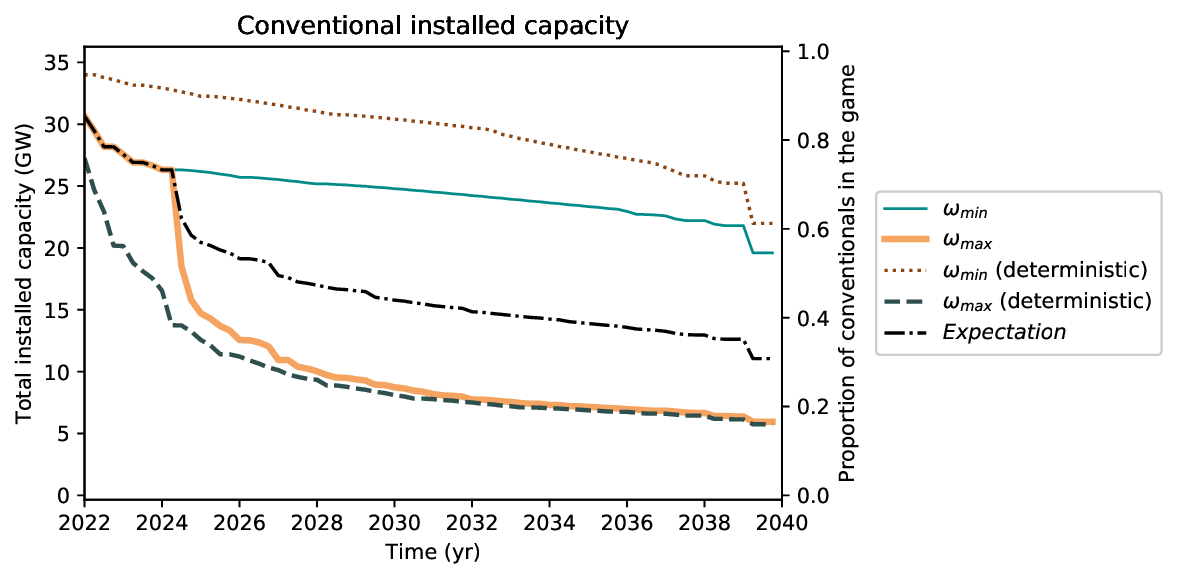}%
        \label{prop_conv 1}%
        }%
    \subfloat{%
        \includegraphics[width=0.435\textwidth]{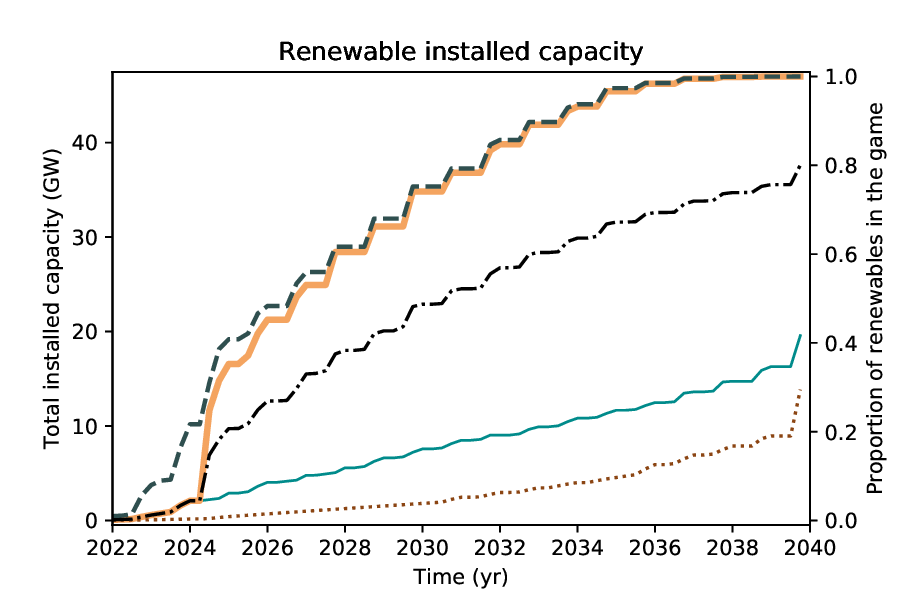}%
        \label{prop_ren 1}%
        }%
        \\
        
        \subfloat{%
        \includegraphics[width=0.565\textwidth]{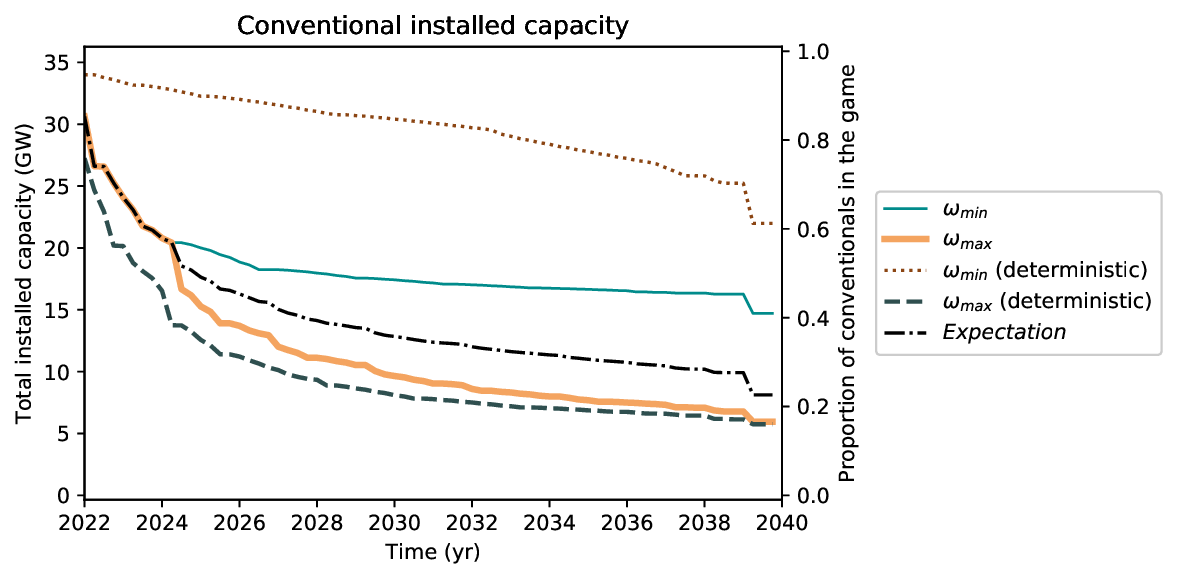}%
        \label{prop_conv 2}%
        }%
    \subfloat{%
        \includegraphics[width=0.435\textwidth]{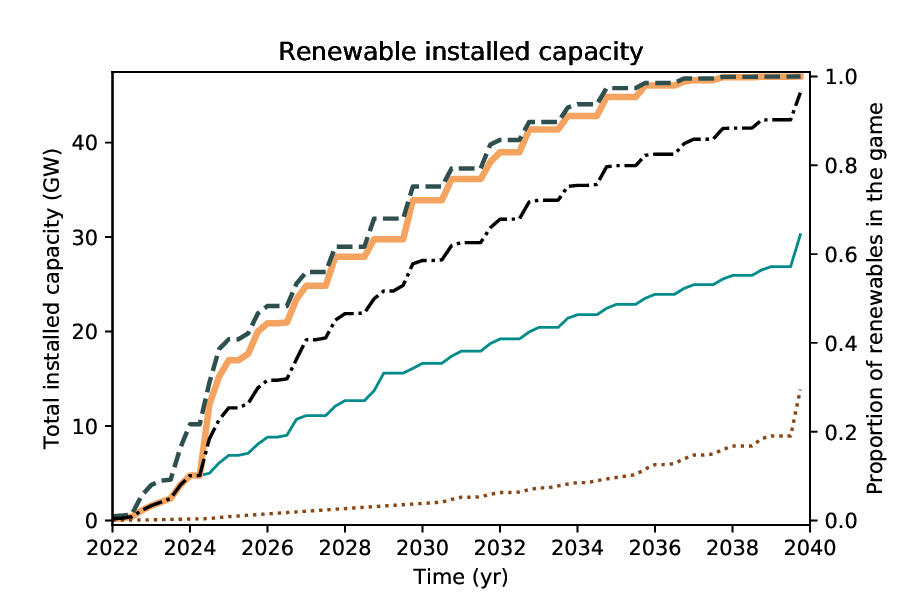}%
        \label{prop_ren 2}%
        }%
    \caption{Total installed capacities and proportions of conventional and renewable producers. Top: Setting 1, bottom: Setting 2.}
    \label{capacity}
\end{figure}

%

\begin{figure}[h]
    \centering
    \subfloat{%
        \includegraphics[width=0.554\textwidth]{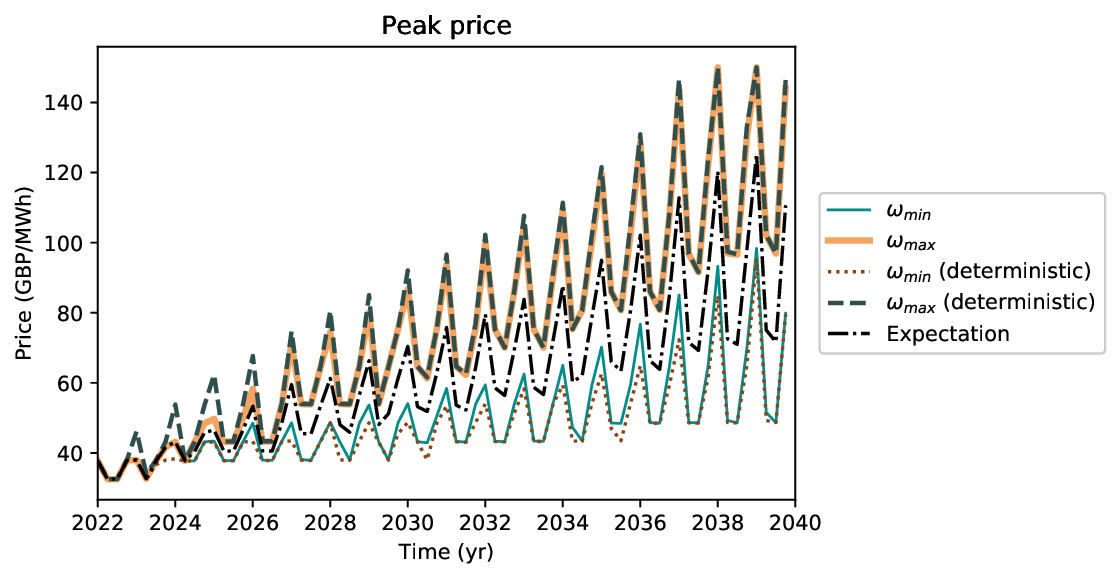}%
        \label{p_price 1}%
        }%
    \hfill%
    \subfloat{%
        \includegraphics[width=0.446\textwidth]{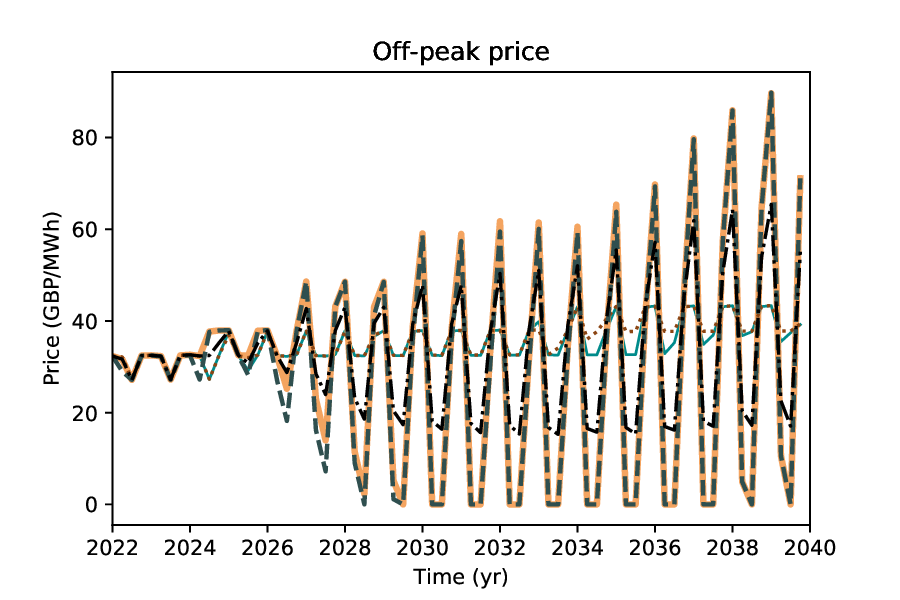}%
        \label{op_price 1}%
        }%
        \\
    \subfloat{%
        \includegraphics[width=0.554\textwidth]{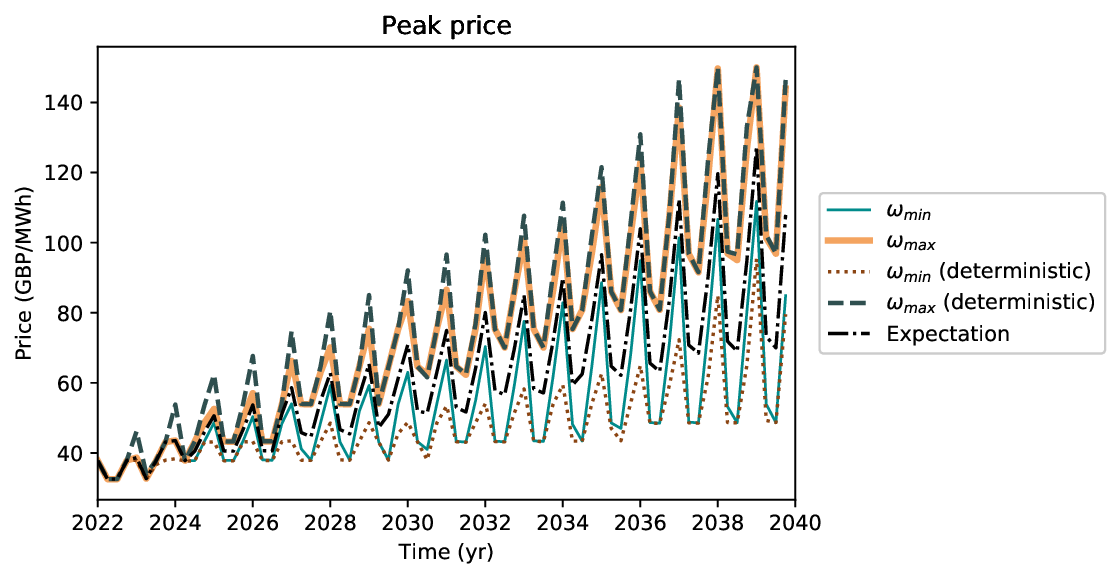}%
        \label{p_price 2}%
        }%
    \hfill%
    \subfloat{%
        \includegraphics[width=0.446\textwidth]{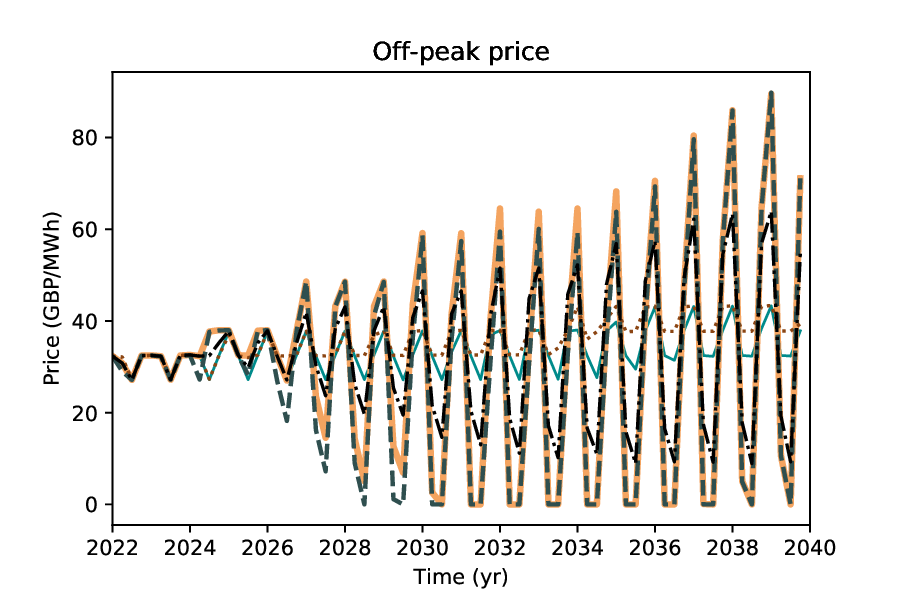}%
        \label{op_price 2}%
        }%
    \caption{Peak and Off-peak prices. Top: Setting 1, bottom: Setting 2.}
    \label{proportions}
\end{figure}

%
%
%

\section*{Statements and declarations}
\subsection*{Competing interests}
The authors have no competing interests to declare that are relevant to the content of this article.
\subsection*{Funding}
Peter Tankov gratefully acknowledges financial support from the ANR (project EcoREES ANR-19-CE05- 0042) from ADEME (Agency for Ecological Transition) in the context of SECRAET project, and from the FIME Research Initiative of the Europlace Institute of Finance.
\subsection*{Data availability}
The electricity demand projections used in this paper are available from \url{https://www.gov.uk/government/publications/updated-energy-and-emissions-projections-2019}, Annex F. The NGFS scenarios used for illustrative purposes (Figure \ref{pricedemand}) are available from \url{https://data.ene.iiasa.ac.at/ngfs/}.
\appendix
\section{Complementary results on the linear programming approach for MFGs of optimal stopping}

In this appendix, we provide a theoretical analysis of the linear programming approach in discrete time in an abstract and general case. We show by a duality argument that the relaxed set is equal to the closed convex hull of the set induced by Markovian stopping times. In particular, this result allows us to derive a probabilistic representation of the relaxed set in terms of randomized stopping times (in the sense  of \cite{carmona2017}).

\paragraph{Markovian and randomized stopping times in the canonical space}

We consider a discrete time setting with horizon $T\in \mathbb N^*$ and we let $I=\{0, \ldots, T\}$ be the set of time indices and $I^*=\{1, \ldots, T\}$. We are given a nonempty compact metric space $(E, d)$ and consider two canonical spaces: $\Omega_X=E^I$ and $\overline \Omega = E^I\times I$, where $E^I$ is the set of functions from $I$ to $E$ endowed with the pointwise convergence topology (which coincides with the uniform convergence topology here). Denote by $\mathcal{B}(\Omega_X)$ and $\mathcal{B}(\overline \Omega)$ the respective Borel $\sigma$-algebras. Let $X$ be the identity map on $\Omega_X$ and let $\overline X$ and $\overline \tau$ be the projections on $\overline\Omega$, i.e.
$\overline X(\overline \omega)=x$ and $\overline \tau(\overline \omega)=\theta$ for each $\overline\omega=(x, \theta)\in \overline \Omega$. On $(\Omega_X, \mathcal{B}(\Omega_X))$ we define the filtration $\mathcal{F}^X_t:=\sigma(X_s:s\leq t)$, for $t\in I$, and on $(\overline \Omega, \mathcal{B}(\overline\Omega))$ we define the filtrations $\mathcal{F}^{\overline X}_t:=\sigma(\overline X_s:s\leq t)$ and $\mathcal{F}^{\overline \tau}_t:=\sigma(\{\overline\tau=0\}, \ldots, \{\overline\tau=t\})$, for $t\in I$. Note that $\mathbb F^{\overline\tau}$ is the smallest filtration for which $\overline\tau$ is a stopping time. Define $\overline{\mathcal{F}}_t:=\mathcal{F}^{\overline X}_t\vee \mathcal{F}^{\overline\tau}_t$, for $t\in I$. For any probability $\op$ on $(\overline\Omega, \mathcal{B}(\overline\Omega))$, we denote by $\mathbb F^{\overline X, \op}=(\mathcal{F}_t^{\overline X, \op})_{t\in I}$ and $\overline{\mathbb F}^{\op}=(\overline{\mathcal F}^{\op}_t)_{t\in I}$ the filtrations
$$\mathcal{F}_t^{\overline X, \op}:=\mathcal{F}_t^{\overline X}\vee \mathcal{N}_{\op}(\overline{\mathcal{F}}_T), \quad \overline{\mathcal{F}}_t^{\op}:=\overline{\mathcal{F}}_t \vee \mathcal{N}_{\op}(\overline{\mathcal{F}}_T).$$
We are given transition kernels $(\pi_t)_{t\in I^*}$ on $E$ and an initial law $m_0^*$ on $E$. We assume that the transition kernels are continuous as maps from $E$ to $\mathcal{P}(E)$. Let $\nu\in \mathcal{P}(E^I)$ be the unique law of the Markov chain with transition kernels $(\pi_t)_{t\in I^*}$ and initial law $m_0^*$. We denote by $\mathbb F^{X, \nu}=(\mathcal{F}_t^{X, \nu})_{t\in I}$ the filtration
$$\mathcal{F}_t^{X, \nu}:=\mathcal{F}_t^{X}\vee \mathcal{N}_{\nu}(\mathcal{F}_T^X).$$

\begin{definition}
Denote by $\mathcal{T}$ the set of measurable functions $\tau:\Omega_X\rightarrow I$ such that $\tau$ is an $\mathbb F^{X, \nu}$-stopping time.
\end{definition}

\begin{definition}
Let $\mathcal{A}_0$ be the set of probabilities $\op$ on $(\overline\Omega, \mathcal{B}(\overline\Omega))$ such that
\begin{enumerate}[(1)]
\item Under $\op$, $\overline X$ is a Markov chain with initial distribution $m_0^*$ and transition kernels $(\pi_t)_{t\in I^*}$.
\item $\overline \tau$ is an $\mathbb F^{\overline X, \op}$-stopping time.
\end{enumerate}
\end{definition}

\begin{lemma}[\textit{Characterization of $\mathcal{A}_0$ by Markovian stopping times}]\label{disint Markov}
It holds that $$\mathcal{A}_0=\{\nu(dx)\delta_{\tau(x)}(d\theta): \tau\in \mathcal{T}\}.$$
\end{lemma}

\begin{proof}
Let $\op\in \mathcal{A}_0$. By the first condition, the first marginal of $\op$ coincides with $\nu$. Since $\overline\tau$ is $\mathcal F_T^{\overline X, \op}$-measurable, there exists some measurable function $\tau:\Omega_X\rightarrow I$ such that $\op(\overline \tau= \tau(\overline X))=1$. In particular we can disintegrate $\op$ as $\op(dx, d\theta)=\nu(dx)\delta_{\tau(x)}(d\theta)$. Using that $\overline \tau$ is an $\mathbb F^{\overline X, \op}$-stopping time and that $\overline \tau= \tau(\overline X)$ $\op$-a.s., we get that $\tau(\overline X)$ is an $\mathbb F^{\overline X, \op}$-stopping time. Moreover, for any $t\in I$ and $B\in \mathcal F_t^{\overline X, \op}$, it is easy to check that $(X, \tau(X))^{-1}(B)\in \mathcal F_t^{X, \nu}$. In particular, taking $B=\{\tau(\overline{X})\leq t\}$, we get $\{\tau(X)\leq t\}\in \mathcal F_t^{X, \nu}$, implying that $\tau\in \mathcal{T}$.

Now let $\tau\in \mathcal{T}$ and define $\op(dx, d\theta):=\nu(dx)\delta_{\tau(x)}(d\theta)$. Since $\op\circ \overline{X}^{-1}=\nu$, the first condition is verified. Define the set $C:=\{\overline{\tau}=\tau(\overline X)\}\in \overline{\mathcal{F}}_T$ which has probability $1$ under $\op$. We have
\begin{align*}
\{\tau(\overline X)\leq t\} = \{x\in \Omega_X:\tau(x)\leq t\}\times I = (C_1\times I)\cup(C_2\times I),
\end{align*}
where $C_1\in \mathcal{F}_t^X$ and $C_2\in \mathcal{N}_\nu(\mathcal{F}_T^X)$ are such that $\{x\in \Omega_X:\tau(x)\leq t\}=C_1\cup C_2$. We can deduce that $C_1\times I\in \mathcal{F}_t^{\overline{X}}$ and $C_2\times I\in \mathcal{N}_{\op}(\overline{\mathcal{F}}_T)$, so that $\{\tau(\overline X)\leq t\}\in \mathcal{F}_t^{\overline{X}, \op}$. Finally,
$$\{\overline{\tau}\leq t\}=(\{\tau(\overline X)\leq t\} \cap C)\cup(\{\overline{\tau}\leq t\}\cap C^c)\in \mathcal{F}_t^{\overline{X}, \op}.$$

\end{proof}

\noindent Let $\mathcal{L}$ be the generator associated to $(\pi_t)_{t\in I^*}$. For any $\varphi\in C(I\times E)$, define the process $M(\varphi)$ by $M_0(\varphi)=\varphi(0, \overline X_0)$ and
$$M_t(\varphi) = \varphi(t, \overline X_t) - \sum_{s=0}^{t-1}\mathcal{L}(\varphi)(s, \overline X_s), \quad t\in I^*.$$
By taking the conditional expectation, one can show the following lemma.

\begin{lemma}[\textit{Equivalence between the Markov property and the martingale problem}]\label{lemma equiv mart}
Let $\overline{\mathbb G}$ be a filtration including $\mathbb F^{\overline{X}}$ and $\op$ be a probability on $(\overline{\Omega}, \mathcal{B}(\overline{\Omega}))$ such that $\op\circ \overline{X}_0^{-1}=m_0^*$. Then, the following are equivalent:
\begin{enumerate}[(1)]
\item Under $\op$, $\overline X$ is a $\overline{\mathbb G}$-Markov chain with transition kernels $(\pi_t)_{t\in I^*}$.
\item Under $\op$, for any $\varphi\in C(I\times E)$, the process $M(\varphi)$ is a $\overline{\mathbb G}$-martingale.
\end{enumerate}
\end{lemma}

\begin{definition}
Let $\mathcal{A}_1$ be the set of probabilities $\op$ on $(\overline \Omega, \mathcal{B}(\overline \Omega))$ such that under $\op$, $\overline X$ is an $\overline{\mathbb F}$-Markov chain with initial distribution $m_0^*$ and transition kernels $(\pi_t)_{t\in I^*}$. By Lemma \ref{lemma equiv mart}, $\mathcal{A}_1$ is exactly the set of probability measures $\op$ on $(\overline \Omega, \mathcal{B}(\overline \Omega))$ such that $\op\circ \overline X_0^{-1}=m_0^*$ and under $\op$, for any $\varphi\in C(I\times E)$, the process $M(\varphi)$ is an $\overline{\mathbb F}$-martingale.
\end{definition}

\begin{lemma}\label{inclusion strict}
We have the inclusion $\mathcal{A}_0\subset \mathcal{A}_1$.
\end{lemma}

\begin{proof}
Let $\op\in \mathcal{A}_0$. Since $\overline\tau$ is an $\mathbb F^{\overline X, \op}$-stopping time, we have that $\overline{\mathbb F}^{\op}=\mathbb F^{\overline X, \op}$. For any $\varphi\in C(E)$, we obtain
$$\mathbb E^{\op}[\varphi(\overline X_t)|\overline{\mathcal F}_{t-1}]=\mathbb E^{\op}[\varphi(\overline X_t)|\overline{\mathcal F}_{t-1}^{\op}]=\mathbb E^{\op}[\varphi(\overline X_t)|\mathcal F_{t-1}^{\overline X, \op}]=\mathbb E^{\op}[\varphi(\overline X_t)|\overline X_{t-1}].$$

\end{proof}

\begin{lemma}[\textit{Immersion property}]\label{Immersion property}
If $\op \in \mathcal{A}_1$, then $\mathbb F^{\overline X}$ is immersed in $\overline{\mathbb F}$ under $\op$.
\end{lemma}

\begin{proof}
We need to show that for all $t\in I$, under $\op$, $\overline{\mathcal{F}}_t$ is conditionally independent of $\mathcal{F}^{\overline X}_T$ given $\mathcal{F}^{\overline X}_t$. Let $\varphi_t:\overline \Omega\rightarrow \mathbb R$ be bounded and $\overline{\mathcal{F}}_t$-measurable, $\psi_t:E^I\rightarrow \mathbb R$ be bounded and $\mathcal{F}_t^{\overline X}$-measurable, $\psi_{t+}:E^I\rightarrow \mathbb R$ be bounded and $\sigma(\overline X_{t+1}, \ldots, \overline X_T)$-measurable and $\phi_{t}:E^I\rightarrow \mathbb R$ be bounded and $\mathcal{F}_t^{\overline X}$-measurable. We have
\begin{align*}
\mathbb E^{\op}[\varphi_t(\overline X, \overline \tau)\psi_t(\overline X)\psi_{t+}(\overline X)\phi_{t}(\overline X)]& = \mathbb E^{\op}[\varphi_t(\overline X, \overline \tau)\psi_t(\overline X)\mathbb E^{\op}[\psi_{t+}(\overline X)|\overline{\mathcal{F}}_t]\phi_{t}(\overline X)]\\
&= \mathbb E^{\op}[\varphi_t(\overline X, \overline \tau)\psi_t(\overline X)\mathbb E^{\op}[\psi_{t+}(\overline X)|\mathcal{F}_t^{\overline X}]\phi_{t}(\overline X)]\\
&= \mathbb E^{\op}[\mathbb E^{\op}[\varphi_t(\overline X, \overline \tau)|\mathcal{F}_t^{\overline X}]\mathbb E^{\op}[\psi_t(\overline X)\psi_{t+}(\overline X)|\mathcal{F}_t^{\overline X}]\phi_{t}(\overline X)].
\end{align*}
This shows that 
$$\mathbb E^{\op}[\varphi_t(\overline X, \overline \tau)\psi_t(\overline X)\psi_{t+}(\overline X)|\mathcal{F}_t^{\overline X}]=\mathbb E^{\op}[\varphi_t(\overline X, \overline \tau)|\mathcal{F}_t^{\overline X}]\mathbb E^{\op}[\psi_t(\overline X)\psi_{t+}(\overline X)|\mathcal{F}_t^{\overline X}].$$
This is sufficient to prove the claim.

\end{proof}

Applying a discrete time version of Proposition 1.10 in Chapter 1 in \cite{carmona2018b}, we get the following corollary.

\begin{corollary}\label{coro disint}
We have the equality $\mathcal{A}_1=\{\nu(dx)\kappa(x, d\theta):\kappa \in \mathcal{K}\}$, where $\mathcal{K}$ is the set of transition kernels from $\Omega_X$ to $I$ such that for all $t\in I$ and $B\in \sigma(\{0\}, \ldots, \{t\})$, the mapping $x\mapsto \kappa(x, B)$ is $\mathcal{F}_t^{X, \nu}$-measurable.
\end{corollary}

In particular, the set $\mathcal{A}_1$ can be interpreted as the set of randomized stopping times since each kernel $\kappa\in \mathcal{K}$ gives a probability to stop given the observation of the underlying process (in an adapted way through the filtration of the process). The set $\mathcal{A}_0$ corresponds to the set of Markovian stopping times in the sense that up to null sets, the stopping rule is determined by the information given by the underlying Markov process and the kernel is given by a Dirac measure (cf Lemma \ref{disint Markov}).

\begin{proposition}
The set $\mathcal{A}_1$ is compact and convex.
\end{proposition}

\begin{proof}
The relative compactness of $\mathcal{A}_1$ follows since $\mathcal{P}(\overline\Omega)$ is compact. Let us show that $\mathcal{A}_1$ is closed. Let $(\op_n)_{n\geq1}\subset \mathcal{A}_1$ converging weakly to some $\op$. Using the continuity of the map $(x, \theta)\mapsto x_0$ and passing to the limit, we get $\op\circ \overline X_0^{-1}=m_0^*$. Let $\varphi\in C(E)$ and $\psi\in C(\overline \Omega)$ an $\overline{\mathcal{F}}_{t-1}$-measurable function, then
$$\mathbb E^{\op}[(M_t(\varphi)-M_{t-1}(\varphi))\psi(\overline X, \overline \tau)]=\lim_{n\rightarrow\infty}\mathbb E^{\op_n}[(M_t(\varphi)-M_{t-1}(\varphi))\psi(\overline X, \overline \tau)]=0.$$
This shows that $\mathcal{A}_1$ is closed and henceforth compact. The convexity follows by the linearity of the conditions with respect to the probability measure.

\end{proof}

\noindent For any $\op\in \mathcal{A}_1$, we define the occupation measures
$$m_t^{\op}(B):=\op(\overline X_t\in B, t<\overline \tau), \quad B\in \mathcal{B}(E),\quad t\in I\setminus\{T\},$$
$$\mu_t^{\op}(B):=\op(\overline X_t\in B, \overline \tau=t), \quad B\in \mathcal{B}(E), \quad t\in I.$$

\begin{definition}
Define the sets
$$\mathcal{R}_0:=\{(\mu^{\op}, m^{\op}): \op\in \mathcal{A}_0\},\qquad \mathcal{R}_1:=\{(\mu^{\op}, m^{\op}): \op\in \mathcal{A}_1\}.$$
\end{definition}

\begin{proposition}\label{R_1 compact}
The set $\mathcal{R}_1$ is compact and convex.
\end{proposition}

\begin{proof}
For the compactness it suffices to show that $\mathcal{A}_1\ni \op\mapsto (\mu^{\op}, m^{\op})$ is continuous, since $\mathcal{A}_1$ is compact. Let $(\op_n)_{n\geq1}\subset \mathcal{A}_1$ converging weakly to $\op\in \mathcal{A}_1$. For any $\varphi\in C(E)$ and $t\in I\setminus\{T\}$, 
$$\int_E\varphi(x)m_t^{\op_n}(dx)=\mathbb E^{\op_n}[\varphi(\overline X_t)\mathds{1}_{t<\overline\tau}]\underset{n\rightarrow\infty}{\longrightarrow}\mathbb E^{\op}[\varphi(\overline X_t)\mathds{1}_{t<\overline\tau}]=\int_E\varphi(x)m_t^{\op}(dx),$$
where we used the continuity of the function $\overline\Omega\ni (x, \theta)\mapsto\varphi(x_t)\mathds{1}_{t<\theta}$. This shows that for each $t\in I\setminus\{T\}$, $(m_t^{\op_n})_{n\geq 1}$ converges to $m_t^{\op}$. By the same argument, for each $t\in I$, $(\mu_t^{\op_n})_{n\geq 1}$ converges to $\mu_t^{\op}$. We conclude that $\mathcal{R}_1$ is compact. By the linearity of the same mapping $\mathcal{A}_1\ni \op\mapsto (\mu^{\op}, m^{\op})$ and the convexity of $\mathcal{A}_1$, we get that $\mathcal{R}_1$ is convex.

\end{proof}

\noindent We define the relaxed set in this framework as follows.

\begin{definition}
Let $\mathcal{R}$ be the set of pairs $(\mu, m)\in \mathcal{P}^{sub}(E)^I\times \mathcal{P}^{sub}(E)^{I\setminus\{T\}}$ such that for all $\varphi\in C(I\times E)$,
\begin{align*}
\sum_{t=0}^T\int_{E}\varphi(t, x) \mu_t(dx) =\int_{E} \varphi(0, x)m_0^*(dx)
+ \sum_{t=0}^{T-1}\int_{E}\mathcal{L}(\varphi)(t, x) m_t(dx).
\end{align*}
\end{definition}

\begin{proposition}\label{inclusions R}
We have the inclusions $\mathcal{R}_0\subset \mathcal{R}_1\subset \mathcal{R}$. In particular $\overline{\text{conv}}(\mathcal{R}_0)\subset \mathcal{R}_1\subset \mathcal{R}$.
\end{proposition}

\begin{proof}
The first inclusion is a consequence of the inclusion $\mathcal{A}_0\subset\mathcal{A}_1$ proved in Lemma \ref{inclusion strict}. For the second inclusion, let $\op\in \mathcal{A}_1$ and consider the associated measures $(\mu^{\op}, m^{\op})$. Using that for any $\varphi\in C(I\times E)$, $M(\varphi)$ is an $\overline{\mathbb F}$-martingale under $\op$ and $\overline \tau$ is an $\overline{\mathbb F}$-stopping time, we get
$$\mathbb E^{\op}[M_{\overline\tau}(\varphi)|\overline{\mathcal{F}}_0]=\varphi(0, \overline X_0),$$
which implies by taking the expectation and replacing the expression of $M_{\overline\tau}(\varphi)$,
$$\mathbb E^{\op}[\varphi(\overline\tau, \overline X_{\overline\tau})] = \mathbb E^{\op}[\varphi(0, \overline X_0)] + \mathbb E^{\op}\left[\sum_{t=0}^{\overline \tau-1}\mathcal{L}(\varphi)(t, \overline X_t)\right].$$
By the definition of the measures we obtain
$$\sum_{t=0}^T\int_E\varphi(t, x)\mu_t^{\op}(dx)=\int_E\varphi(0, x)m_0^*(dx) + \sum_{t=0}^{T-1}\int_E\mathcal{L}(\varphi)(t, x)m_t^{\op}(dx).$$
The last inclusions follow by Proposition \ref{R_1 compact}.

\end{proof}

\paragraph{Probabilistic representation}

We want to show that the sets $\mathcal{R}$ and $\mathcal{R}_1$ are equal. We will follow similar arguments to \cite{fleming1988} in order to show that the closed convex hull of $\mathcal{R}_0$ is equal to $\mathcal{R}$, which is sufficient to obtain the desired equality according to Proposition \ref{inclusions R}. \vspace{5pt}

\emph{Equality of the values.} We are given two functions $f\in C(I\setminus \{T\}\times E)$ and $g\in C(I\times E)$. We place ourselves on the filtered probability space $(\Omega_X, \mathcal{B}(\Omega_X), \mathbb F^X, \nu)$. For each $(t, x)\in I\times E$, define the probability measure on $\nu_{t, x}\in (\Omega_X, \mathcal{B}(\Omega_X))$ by
$$\nu_{t, x}(dx_0, \ldots, dx_T)=\prod_{s=0}^t\delta_x(dx_s)\prod_{\ell=t+1}^T\pi_\ell(x_{\ell-1}; dx_\ell).$$
Denote by $\mathcal{T}_{t, x}$ the set of $\{t, \ldots, T\}$-valued $\mathbb F^{X, \nu_{t, x}}$-stopping times. Define the value function
$$v(t, x):=\sup_{\tau\in \mathcal{T}_{t, x}}\mathbb E^{\nu_{t, x}}\left[\sum_{s=t}^{\tau-1}f(s, X_s) + g(\tau, X_{\tau})\right].$$
The value function verifies the dynamic programming principle (Theorem 1.9 in Chapter 1 in \cite{peskir2006}):
$$v(T, x)=g(T, x), \quad x\in E,$$
$$v(t, x)=\max\left\{f(t, x) + \int_{E}v(t+1, x')\pi_{t+1}(x, dx'), g(t, x)\right\}, \quad x\in E, \; t\in I\setminus\{T\}.$$
By backward induction, since $g$, $f$ and the transition kernels are continuous, we obtain that $v\in C(I\times E)$. Define the quantities
\begin{align*}
V^S(f, g)&:=\sup_{(\mu, m)\in \mathcal{R}_0}\sum_{t=0}^{T-1}\int_Ef(t, x)m_t(dx) + \sum_{t=0}^T\int_Eg(t, x)\mu_t(dx)
\\&=\sup_{\tau \in \mathcal{T}}\mathbb E^\nu\left[\sum_{t=0}^{\tau-1}f(t, X_t) + g(\tau, X_{\tau})\right],\\
V^{LP}(f, g)&:=\sup_{(\mu, m)\in \mathcal{R}}\sum_{t=0}^{T-1}\int_Ef(t, x)m_t(dx) + \sum_{t=0}^T\int_Eg(t, x)\mu_t(dx).    
\end{align*}
By Proposition \ref{inclusions R}, we have $V^S\leq V^{LP}$.
The Snell envelope $Y=(Y_t)_{t\in I}$ associated to $(f, g)$ is $Y_t=v(t, X_t)+\sum_{s=0}^{t-1}f(s, X_s)$. We say that $\tau^\star\in \mathcal{T}$ is an $(f, g)$-optimal stopping time if 
$$V^S(f, g)=\mathbb E^\nu\left[\sum_{t=0}^{\tau^\star-1}f(t, X_t) + g(\tau^\star, X_{\tau^\star})\right].$$
The random variable $\tau_{min}:=\inf\{t\in \{0, \ldots, T\}: v(t, X_t) =g(t, X_t)\}$ is an $(f, g)$-optimal stopping time and $(Y_{t\wedge \tau_{min}})_{t\in I}$ is an $\mathbb F^X$-martingale. In particular,
$$V^S(f, g)=\mathbb E^\nu[Y_{\tau_{min}}]=\mathbb E^\nu[v(0, X_0)]=\int_E v(0, x)m_0^*(dx).$$
Now, for all $(\mu, m)\in \mathcal{R}$, using $v$ as a test function in the constraint, we get
\begin{align*}
V^S(f, g)&=\sum_{t=0}^T\int_E v(t, x)\mu_t(dx) -\sum_{t=0}^{T-1}\int_E \mathcal{L}(v)(t, x)m_t(dx) \\
&\geq \sum_{t=0}^T\int_E g(t, x)\mu_t(dx) +\sum_{t=0}^{T-1}\int_E f(t, x)m_t(dx),
\end{align*}
where the inequality is a consequence of the dynamic programming principle. Taking the supremum over $(\mu, m)\in \mathcal{R}$, we deduce that $V^S(f, g)\geq V^{LP}(f, g)$, which implies that $V^S(f, g)= V^{LP}(f, g)$.

\begin{theorem}\label{proba rep}
The set $\mathcal{R}$ is equal to the closed convex hull of $\mathcal{R}_0$. As a consequence, $\mathcal{R}=\mathcal{R}_1$, meaning that we can represent any $(\mu, m)\in \mathcal{R}$ with some randomized stopping time.
\end{theorem}

\begin{proof}
Assume that there exists some $(\mu^0, m^0)\in \mathcal{R}\setminus \overline{\text{conv}}(\mathcal{R}_0)$. Recall that for any compact metric space $K$, the topological dual of $\mathcal{M}_s(K)$ endowed with the weak topology $\sigma(\mathcal{M}_s(K), C(K))$ is $C(K)$. In particular, by Theorem 3.4 (b) p.59 in \cite{rudin1991}, there exists some $(f^0, g^0)\in C(I\setminus\{T\}\times E)\times C(I\times E)$ and $c\in \mathbb R$ such that for all $(\mu, m)\in \overline{\text{conv}}(\mathcal{R}_0)$,
\begin{multline*}
\sum_{t=0}^{T-1}\int_E f^0(t, x)m_t(dx) + \sum_{t=0}^T\int_E g^0(t, x)\mu_t(dx)\\<c
<\sum_{t=0}^{T-1}\int_E f^0(t, x)m_t^0(dx) + \sum_{t=0}^T\int_E g^0(t, x)\mu_t^0(dx)
\leq V^{LP}(f^0, g^0).
\end{multline*}
Taking the supremum over $(\mu, m)\in \mathcal{R}_0$, we get $V^{S}(f^0, g^0)<V^{LP}(f^0, g^0)$, which is a contradiction.

\end{proof}

\section{Other technical results}

\subsection{Exchangeable random variables and De Finetti’s theorem}

We recall a version of De Finetti's theorem adapted to our setting. Let $(E, d)$ and $(F, \rho)$ be two complete and separable metric spaces. Consider a complete probability space $(\Omega, \mathcal{F}, \mathbb P)$, a sequence of $E$-valued random variables $(X_n)_{n\geq 1}$ and an $F$-valued random variable $Y$. 

\begin{definition}\label{def cond iid}
We say that $(X_n)_{n\geq 1}$ is i.i.d. given $Y$ if the following two conditions are satisfied.
\begin{enumerate}[(1)]
\item For all $n\geq 1$ and all $(B_k)_{k\leq n}\in \mathcal{B}(E)^n$, 
$$\mathbb P\left(\cap_{k=1}^n\{X_k\in B_k\}|Y\right)=\prod_{k=1}^n\mathbb P(X_k\in B_k|Y)\quad a.s.$$
\item For all $n, k\geq 1$ and all $B\in \mathcal{B}(E)$, $\mathbb P(X_n\in B|Y)=\mathbb P(X_k\in B|Y)$ a.s.
\end{enumerate}
\end{definition}

If $(X_n)_{n\geq 1}$ is i.i.d. given $Y$, then one can show that the sequence $(X_n, Y)_{n\geq 1}$ is exchangeable. In particular, using the same argument as in \cite{kingman1978}, we obtain that for all measurable functions $\psi:E\times F\mapsto \mathbb R$ such that $\mathbb E[|\psi(X_1, Y)|]<\infty$,
\begin{equation}\label{exchangeable}
\lim_{n\rightarrow\infty}\frac{1}{n}\sum_{k=1}^n\psi(X_k, Y)=\mathbb E[\psi(X_1, Y)|\mathcal{G}]\quad a.s.    
\end{equation}
where $\mathcal{G}$ is the exchangeable $\sigma$-algebra. Using Theorem 3 in \cite{olshen1974}, one can replace $\mathcal{G}$ in \eqref{exchangeable} by $\sigma(Y)$ in order to derive the following result.

\begin{theorem}\label{clln}
If $(X_n)_{n\geq 1}$ is i.i.d. given $Y$, then, for all measurable functions $\psi:E\times F\mapsto \mathbb R$ such that $\mathbb E[|\psi(X_1, Y)|]<\infty$, we have
$$\lim_{n\rightarrow\infty}\frac{1}{n}\sum_{k=1}^n\psi(X_k, Y)=\mathbb E[\psi(X_1, Y)|Y]\quad a.s.$$
\end{theorem}

\subsection{Dynkin's formula for randomized stopping times}

\begin{proposition}\label{Dynkin rand}
Let $\varphi\in C(I\times E)$ and $\kappa$ be a randomized stopping time. Then,
$$\sum_{t=0}^T\mathbb E[\varphi(t, X_t)\kappa(\{t\})]=\mathbb E[\varphi(0, X_0)] + \sum_{t=0}^{T-1}\mathbb E[\mathcal{L}(\varphi)(t, X_t)\kappa(\{t+1, \ldots, T\})].$$
\end{proposition}

\begin{proof}
Let $\overline{\Omega}:=\Omega\times I$, $\overline{\mathcal{F}}:=\mathcal{F}\otimes 2^I$, $\overline{\mathbb P}(d\omega, ds):=\mathbb P(d\omega)\kappa(\omega, ds)$. Consider $\overline X_t(\omega, s):=X_t(\omega)$ and $\theta(\omega, s):= s$. Define the filtration $\mathbb F^\theta:=(\mathcal{F}^\theta_t)_t$ with $\mathcal{F}_t^\theta:=\sigma(\{\theta=0\}, \ldots, \{\theta=t\})$. Now, construct the filtration $\overline{\mathbb F}:=(\overline{\mathcal{F}}_t)_t$, with $\overline{\mathcal{F}}_t:=\mathcal{F}_t\otimes \mathcal{F}^\theta_t$. By construction $\theta$ is an $\overline{\mathbb F}$-stopping time. We are going to show that $\overline{X}:=(\overline{X}_t)_t$ is an $\overline{\mathbb F}$-Markov chain. Let $\psi\in C(E)$. We show first that 
$$\overline{\mathbb E}[\psi(\overline{X}_{t})|\overline{\mathcal{F}}_{t-1}]=\mathbb E[\psi(X_t)|\mathcal{F}_{t-1}],\quad \overline{\mathbb P}-a.s.$$
Let $B\in \mathcal{F}_{t-1}$ and $C\in \mathcal{F}_{t-1}^\theta$. We have
\begin{align*}
\overline{\mathbb E}[\psi(\overline{X}_{t})\mathds{1}_{B\times C}]&=\mathbb E[\psi(X_t)\mathds{1}_{B}\kappa(C)]=\mathbb E[\mathbb E[\psi(X_t)|\mathcal{F}_{t-1}]\mathds{1}_{B}\kappa(C)]=\overline{\mathbb E}[\mathbb E[\psi(X_t)|\mathcal{F}_{t-1}]\mathds{1}_{B\times C}],
\end{align*}
where in the second equality we used the measurability property of $\kappa$. We finally have
$$\overline{\mathbb E}[\psi(\overline{X}_{t})|\overline{\mathcal{F}}_{t-1}]=\mathbb E[\psi(X_t)|\mathcal{F}_{t-1}]=\int_E\psi(x)\pi_{t}(\overline{X}_{t-1}; dx).$$
This shows that $\overline{X}$ is an $\overline{\mathbb F}$-Markov chain. Now, by Dynkin's formula,
$$\overline{\mathbb E}[\varphi(\theta, \overline{X}_{\theta})]=\overline{\mathbb E}[\varphi(0, \overline{X}_0)] + \overline{\mathbb E}\left[\sum_{t=0}^{\theta-1}\mathcal{L}(\varphi)(t, \overline X_t)\right].$$
By applying Fubini's theorem we obtain the desired result.

\end{proof}

\subsection{A continuity property for the price function}

\begin{lemma}\label{price cont}
Let $(E, d)$ be a metric space. Consider the function $P:E\rightarrow \mathbb R_+$ defined by
$$P(x):=\inf\{p\geq 0:r(x)\leq s(x, p)\}\wedge p_{max}, \quad x\in E,$$
where $r:E\rightarrow \mathbb R_+$ and $s:E\times \mathbb R_+\rightarrow\mathbb R_+$ are continuous, for each $x\in E$, $p\mapsto s(x, p)$ is increasing with $s(x, 0)=0$ and there exists $\tilde p\geq 0$ such that $\inf_{x\in E}s(x, \tilde p)\geq \sup_{x\in E}r(x)$. Then there exists for each $x\in E$ a unique $p^\star(x)\in [0, \tilde p]$ such that
$$r(x)=s(x, p^\star(x)), \quad P(x)=p^\star(x)\wedge p_{max}, \quad x\in E.$$
Moreover, $x\mapsto p^\star(x)$ is continuous and as a consequence $P$ is also continuous. 
\end{lemma}

\begin{proof}
For each $x\in E$, since $r(x)\leq s(x, \tilde p)$, we have that $\tilde p\in \{p\geq 0:r(x)\leq s(x, p)\}$. Using that $s(x, 0)=0\leq r(x)$ and $p\mapsto s(x, p)$ is continuous and increasing we must have that there exists a unique $p^\star(x)\in [0, \tilde p]$ such that $r(x)=s(x, p^\star(x))$ and for all $p<p^\star(x)$, $r(x)>s(x, p^\star(x))$. This allows us to write $P(x)=p^\star(x)\wedge p_{max}$. Let us show that $p^\star(x)$ is continuous. Let $(x_n)_{n\geq 1}\subset E$ converging to $x\in E$. Since $[0, \tilde p]$ is compact, there exists at least one limit point of the sequence $(p^\star(x_n))_{n\geq 1}$. Let $p_0$ be a limit point of the above sequence and assume without loss of generality that the whole sequence converges to that point. Using the continuity of $r$ and $s$ and passing to the limit in the equality $r(x_n)=s(x_n, p^\star(x_n))$, we obtain $r(x)=s(x, p_0)$, and henceforth $p_0=p^\star(x)$.

\end{proof}

\printbibliography

\end{document}